\documentclass{amsart}
\usepackage{graphicx}

\usepackage{amsfonts}
\usepackage{amssymb}
\usepackage{amsmath}
\usepackage{amsthm}
\usepackage{amscd}

\def\vbar{\mathchoice{\vrule height6.3ptdepth-.5ptwidth.8pt\kern- .8pt}
{\vrule height6.3ptdepth-.5ptwidth.8pt\kern-.8pt} {\vrule
height4.1ptdepth-.35ptwidth.6pt\kern-.6pt} {\vrule
height3.1ptdepth-.25ptwidth.5pt\kern-.5pt}}
\def\fudge{\mathchoice{}{}{\mkern.5mu}{\mkern.8mu}}
\def\bbc#1#2{{\rm \mkern#2mu\vbar\mkern-#2mu#1}}
\def\bbb#1{{\rm I\mkern-3.5mu #1}}
\def\bba#1#2{{\rm #1\mkern-#2mu\fudge #1}}
\def\bb#1{{\count4=`#1 \advance\count4by-64 \ifcase\count4\or\bba
A{11.5}\or \bbb B\or\bbc C{5}\or\bbb D\or\bbb E\or\bbb F \or\bbc
G{5}\or\bbb H\or \bbb I\or\bbc J{3}\or\bbb K\or\bbb L \or\bbb
M\or\bbb N\or\bbc O{5} \or \bbb P\or\bbc C{5}\or\bbb B\or\bbc
S{4.2}\or\bba T{10.5}\or\bbc U{5}\or \bba V{12}\or\bba
W{16.5}\or\bba X{11}\or\bba Y{11.7}\or\bba Z{7.5}\fi}}

\newcommand{\K}{{\mathbb{K}}}
\newcommand{\A}{{\mathcal{A}}}

\newtheorem{df}{Definition}[section]
\newtheorem{thm}{Theorem}[section]
\newtheorem{cor}{Corollary}[section]
\newtheorem{rem}{Remark}[section]

\newtheorem{prop}{Proposition}[section]
\newtheorem{exa}{Example}[section]
\newtheorem{lem}{Lemma}[section]

\begin{document}
\date{}
\title{Hom-Alternative, Hom-Malcev and Hom-Jordan superalgebras}
\author{K. Abdaoui, F. Ammar and A. Makhlouf }
\address{K. Abdaoui and  F. Ammar, University of Sfax, Faculty of Sciences Sfax,  BP
1171, 3038 Sfax, Tunisia. }
\email{Abdaouielkadri@hotmail.com}
\email{Faouzi.Ammar@fss.rnu.tn}
\address{A. Makhlouf University of Haute Alsace, 4 rue des fr\`eres Lumi\`ere, 68093 Mulhouse France. }\email{Abdenacer.Makhlouf@uha.fr}
 \maketitle{}
\begin{abstract} Hom-alternative, Hom-Malcev and Hom-Jordan superalgebras are $\mathbb{Z}_{2}$-graded generalizations of Hom-alternative, Hom-Malcev and Hom-Jordan algebras, which are Hom-type generalizations of alternative, Malcev and Jordan algebras. In this paper we prove that Hom-alternative superalgebras are Hom-Malcev-admissible and are also Hom-Jordan-admissible. Home-type generalizations of some well known identities in alternative superalgebras, including the $\mathbb{Z}_{2}$-graded Bruck-Kleinfled function are obtained.
\end{abstract}
\section{Introduction}
A Malcev superalgebra is a non-associative superalgebra $\mathcal{A}$ with a super skewsymmetric multiplication $[-,-]$ $($i.e, $[x,y]=-(-1)^{|x||y|}[y,x])$ such that the Malcev super-identity
\begin{eqnarray}\label{MALCEV-IDI}
2[t,\mathcal{J}_{\mathcal{A}}(x,y,z)]&=&\mathcal{J}_{\mathcal{A}}(t,x,[y,z])+(-1)^{|x|(|y|+|z|)}\mathcal{J}_{\mathcal{A}}(t,y,[z,x])\nonumber\\
&+&(-1)^{|z|(|x|+|y|)}\mathcal{J}_{\mathcal{A}}(t,z,[x,y])
\end{eqnarray}
is satisfied for all homogeneous elements $x,y,z,t$ in the superspace $\mathcal{A}$, where
\begin{equation}\label{sup-Jacobian}
 \mathcal{J}_{\mathcal{A}}(x,y,z)=[[x,y],z]-[x,[y,z]]-(-1)^{|y||z|}[[x,z],y]
 \end{equation}
  is the super-Jacobian.
In particular, Lie superalgebras are examples of Malcev superalgebras. Malcev superalgebras play an important role in the geometry of smooth loops.\\
Closely related to Malcev superalgebras are alternative superalgebras. An alternative superalgebras \cite{Zelm and she-alter} is a superalgebras whose associator is a super-alternating function. In particular, all associative superalgebras are alternative.\\
A Jordan superalgebra is a super-commutative superalgebra $($i.e. $x\cdot y=(-1)^{|x||y|}y\cdot x)$ that satisfies the Jordan super-identity
\begin{eqnarray}\label{SUPER-IDE-JORDAN}
\sum \limits _{x,y,t}(-1)^{|t|(|x|+|z|)} as_{\mathcal{A}}(x\cdot y,z,t)&=&0,
\end{eqnarray}
where $\sum \limits _{x,y,t}$ denotes the cyclic sum over $(x,y,t)$ and $as_{\mathcal{A}}(x,y,z)=(x\cdot y)\cdot z- x\cdot (y\cdot z)$ for all homogeneous elements $x,y,z \in \mathcal{A}$. \\
Starting with an alternative superalgebra $\mathcal{A}$, it is known that the  super-product
$$ x \ast y=\frac{1}{2}(x\cdot y+(-1)^{|x||y|}y\cdot x)$$
gives a Jordan superalgebra $\mathcal{A}^{+}=(\mathcal{A},\ast)$. In other words, alternative superalgebras are Jordan-admissible. The reader is referred to \cite{cantarini2007-jor-poisson,shest-jordan-super,shest-poisson} for discussions about the important role of Jordan superalgebras in physics, especially in quantum mechanics.\\
The purpose of this paper is to study Hom-type generalizations of alternative superalgebras, Malcev $($-admissible$)$ superalgebras
and Jordan $($-admissible$)$ superalgebras.\\
In Section $2$, we introduce Hom-alternative superalgebras and prove two construction results Theorems $(\ref{induced A})$ and $(\ref{derived A})$. Theorem $(\ref{induced A})$ says that the category of Hom-alternative superalgebras is closed under self weak morphism.\\
In Section $3$, we introduce Hom-Malcev superalgebra and prove two construction Theorems $(\ref{inducedMALCEV})$ and $(\ref{derivMALCEV})$. Hom-Malcev superalgebras include Malcev superalgebras and Hom-Lie superalgebras as examples. Theorem $(\ref{inducedMALCEV})$ says that the category of Hom-Malcev superalgebras is closed under self weak morphism.\\
In Section $4$, we show that Hom-alternative superalgebras are Hom-Malcev-admissible $($Theorem $(\ref{hAandMalcev admissible}))$. That is, the super-commutator Hom-superalgebra $($Definition $(\ref{supercommutator}))$ of a Hom-alternative superalgebra is a Hom-Malcev superalgebra, generalizing the fact that alternative superalgebras are Malcev-admissible. The proof of the Hom-Malcev-admissibility of Hom-alternative superalgebras involves the Hom-type analogues of certain identities that holds in alternative superalgebra and of the $\mathbb{Z}_{2}$-graded Bruck-Kleinfled function.\\
In Section $5$, we introduce Hom-flexible superalgebras and we consider the class of Hom-Malcev-admissible superalgebras. In Proposition $(\ref{HMALCEV AND FLEX})$, we give several characterizations of Hom-Malcev-admissible superalgebras that are also Hom-flexible. Hom-alternative superalgebras are Hom-flexible, so by Theorem $(\ref{hAandMalcev admissible})$ Hom-alternative superalgebras are both Hom-flexible and Hom-Malcev-admissible. In Examples $(\ref{exp1-flex})$ and $(\ref{exp2-flexi})$, we construct Hom-flexible, Hom-Malcev-admissible superalgebras that are not Hom-alternative, not Hom-Lie-admissible, and not Malcev-admissible. \\
In Section $6,$ we introduce and study Hom-Jordan $(-$admissible$)$ superalgebras, which are the Hom-type generalizations of Jordan $(-$admissible $)$ superalgebras. We show that Hom-alternative superalgebras are Hom-Jordan-admissible $($Theorem $(\ref{has}))$. In other words, the plus Hom-superalgebra
$($Definition $(\ref{plus hom}))$ of any Hom-alternative superalgebra is a Hom-Jordan superalgebra, generalizing the Jordan-admissibility of alternative superalgebras. Construction results analogous of Theorems $(\ref{derived A})$ and $(\ref{induced A})$ are provided for Hom-Jordan $(-$admissible $)$ superalgebras
$($Theorems $(\ref{iduced})$ and $(\ref{admissible}))$. In Examples $(\ref{ex-jordan1})$ and $(\ref{ex-jorda2})$, we construct two $($non-Jordan$)$ Hom-Jordan superalgebras using the $3$-dimensional Kaplansky superalgebra, respectively the $4$-dimensional simple Jordan superalgebra $D_{t}$, where $t\neq 0$.
\begin{center}\section{Hom-Alternative Superalgebras}\end{center}
Throughout this paper $\mathbb{K}$ is an algebraically closed field of characteristic $0$ and $\mathcal{A}$ is a linear super-espace over $\mathbb{K}$.
In this section, we introduce Hom-alternative superalgebras and study their general properties. We provide some construction results for Hom-alternative superalgebras $($Theorem $(\ref{derived A})$ and Theorem $(\ref{induced A})$).\\
Now let $\mathcal{A}$ be a linear superespace over $\mathbb{K}$ that is a  $\mathbb{Z}_{2}$-graded linear space with a direct sum
$\mathcal{A}=\mathcal{A}_{0}\oplus \mathcal{A}_{1}$. The element of $\mathcal{A}_{j}$, $j \in \mathbb{Z}_{2}$, are said to be homogeneous of parity $j$. The
parity of a homogeneous element $x$ is denoted by $|x|$.
In the sequel, we will denote by $\mathcal{H}(\mathcal{A})$ the set of all  homogeneous elements of $\mathcal{A}$.\\
\begin{df} By a Hom-superalgebra we mean a triple $(\mathcal{A},\mu,\alpha)$ in which $\mathcal{A}$ is a $\mathbb{K}-$super-module,
$\mu:\mathcal{A}\times \mathcal{A}\longrightarrow \mathcal{A}$ is an even bilinear map, and $\alpha: \mathcal{A}\longrightarrow \mathcal{A}$ is an even linear map such that $\alpha \circ \mu=\mu \circ \alpha^{\otimes 2}$ $($multiplicativity$)$.
\end{df}
\begin{rem} The multiplicativity of the twisting even map $\alpha$ is built into our definitions of Hom-superalgebra. We chose to impose multiplicativity because many of our results depend on it and all of our concrete examples of Hom-alternative, Hom-Malcev $(-$admissible$)$ and Hom-Jordan $(-$admissible$)$ superalgebras have this property.
\end{rem}
\begin{df} Let $(\mathcal{A},\mu,\alpha)$ be a Hom-superalgebra, that is a $\K$-vector superspace $\A$ together with a multiplication $\mu$ and a linear self-map $\alpha$.\\
$(1)$~~The \textsf{Hom-associator} of $\mathcal{A}$ \cite{makhloufand silvestrov2006hom} is the trilinear map $\widetilde{as_{\mathcal{A}}}:\mathcal{A}\times \mathcal{A}\times \mathcal{A}\longrightarrow \mathcal{A}$ defined as
\begin{equation}\label{ass}
    \widetilde{as_{\mathcal{A}}}=\mu \circ (\mu\otimes \alpha- \alpha \otimes \mu).
\end{equation}
In terms of elements, the map $\widetilde{as_{\mathcal{A}}}$ is given by
$$ \widetilde{as_{\mathcal{A}}}(x,y,z)=\mu(\mu(x,y),\alpha(z))-\mu(\alpha(x),\mu(y,z)).$$
$(2)$~~The \textsf{Hom-super-Jacobian} of $\mathcal{A}$ \cite{ammar2010hom} is the trilinear map  $\widetilde{J}_{\mathcal{A}}:\mathcal{A}\times \mathcal{A}\times \mathcal{A}\longrightarrow \mathcal{A}$ defined as
\begin{equation}\label{Hom-super-Jac}
    \widetilde{J}_{\mathcal{A}}(x,y,z)=\mu(\mu(x,y),\alpha(z))-\mu(\alpha(x),\mu(y,z))-(-1)^{|y||z|}\mu(\mu(x,z),\alpha(y)).
\end{equation}
\end{df}
Note that when $(\mathcal{A},\mu)$ is a superalgebra $($with $\alpha=Id )$, its Hom-associator and Hom-super-Jacobian coincide with its usual associator and super-Jacobian, respectively.
\begin{df}  A \textsf{left Hom-alternative superalgebra} $($resp. \textsf{right Hom-alternative superalgebra}$)$ is a triple $(\mathcal{A},\mu,\alpha)$ consisting of $\mathbb{Z}_{2}-$graded vector space $\mathcal{A}$, an even bilinear map $\mu:\mathcal{A}\times\mathcal{A}\longrightarrow \mathcal{A}$ and an even homomorphism $\alpha:\mathcal{A}\longrightarrow \mathcal{A}$ satisfying the \textsf{left Hom-alternative super-identity}, that is for all $x,y \in \mathcal{H}(\mathcal{A})$,
\begin{equation}\label{sa1}
\widetilde{as_{\mathcal{A}}}(x,y,z)+(-1)^{|x||y|}\widetilde{as_{\mathcal{A}}}(y,x,z)=0,
\end{equation}
respectively, \textsf{right Hom-alternative super-identity}, that is
\begin{equation}\label{sa2}
\widetilde{as_{\mathcal{A}}}(x,y,z)+(-1)^{|y||z|}\widetilde{as_{\mathcal{A}}}(x,z,y)=0.
\end{equation}
A \textsf{Hom-alternative} superalgebra is one which is both left and right Hom-alternative superalgebra. In particular, if $\alpha$ is a morphism of alternative superalgebras $($ i.e, $\alpha \circ \mu=\mu \circ \alpha^{\otimes 2})$, then we call $(\mathcal{A},\mu,\alpha)$ a \textsf{multiplicative} Hom-alternative superalgebra.
\end{df}
Observe that when $\alpha=Id$, the left Hom-alternative super-identity $(\ref{sa1})$ $($resp. right Hom-alternative super-identity $(\ref{sa2}))$  reduces to the usual left alternative super-identity $($resp. right alternative super-identity$)$.
\begin{df} Let $(\mathcal{A},\mu,\alpha)$ and $(\mathcal{A}^{'},\mu^{'},\alpha^{'})$ be two Hom-alternative superalgebras. An even linear map $f:\mathcal{A}\longrightarrow \mathcal{A}^{'}$ is called$:$
\begin{enumerate}
\item a \textsf{weak morphism} of Hom-alternative superalgebras if it is satisfies $f \circ \mu= \mu^{'}\circ (f \otimes f)$.
\item a  is \textsf{morphism} of Hom-alternative superalgebras if $f$ is a weak morphism and $f\circ \alpha=\alpha^{'}\circ f$.
\end{enumerate}
\end{df}
\begin{lem}\label{propriet ass H-A} Let $(\mathcal{A},\mu,\alpha)$ be a Hom-alternative superalgebra. Then
$$\widetilde{as_{\mathcal{A}}}(x,y,z)=-(-1)^{|y||z|+|x||z|+|y||x|}\widetilde{as_{\mathcal{A}}}(z,y,x),$$
for all $x , y , z \in \mathcal{H}(\mathcal{A})$.
\end{lem}
\begin{proof} Since $(\mathcal{A},\mu,\alpha)$ is a Hom-alternative superalgebra. Then, for all $x,y,z$ in $\mathcal{H}(\mathcal{A})$, we have
 $$\widetilde{as_{\mathcal{A}}}(x,y,z)=-(-1)^{|x||y|}\widetilde{as_{\mathcal{A}}}(y,x,z)=-(-1)^{|y||z|}\widetilde{as_{\mathcal{A}}}(x,z,y).$$
So
\begin{eqnarray*}
  \widetilde{as_{\mathcal{A}}}(x,y,z) &=&-(-1)^{|y||z|}\widetilde{as_{\mathcal{A}}}(x,z,y)\\
  &=& (-1)^{|y||z|+|x||z|}\widetilde{as_{\mathcal{A}}}(z,x,y)\\
   &=&-(-1)^{|y||z|+|x||z|+|x||y|}\widetilde{as_{\mathcal{A}}}(z,y,x).
\end{eqnarray*}
\end{proof}
\begin{prop} Let $(\mathcal{A},\mu,\alpha)$ be a Hom-alternative superalgebra. Suppose that $\mu(x,y)=-(-1)^{|x||y|}\mu(y,x)$ for all $x,y \in \mathcal{H}(\mathcal{A})$, then
$$\mu(\alpha(x),\mu(y,z))=-(-1)^{|x||y|}\mu(\alpha(y),\mu(x,z))$$
and
$$\mu(\mu(z,x),\alpha(y))=-(-1)^{|x||y|}\mu(\alpha(y),\mu(z,x)).$$
\end{prop}
\begin{proof} Let $x,y \in \mathcal{H}(\mathcal{A})$, the left Hom-alternative super-identity gives\\
$$\widetilde{as_{\mathcal{A}}}(x,y,z)+ (-1)^{|x||y|}\widetilde{as_{\mathcal{A}}}(y,x,z)=0.$$
Then
$$\mu(\alpha(x),\mu(y,z))-\mu(\mu(x,y),\alpha(z))+(-1)^{|x||y|}\mu(\alpha(y),\mu(x,z))
-(-1)^{|x||y|}\mu(\mu(y,x),\alpha(z))=0,$$
or
\begin{align*}
-\mu(\mu(x,y),\alpha(z))-(-1)^{|x||y|}\mu(\mu(y,x),\alpha(z))
&=(-1)^{|x||y|}\mu(\mu(y,x),\alpha(z))-(-1)^{|x||y|}\mu(\mu(y,x),\alpha(z))\\
&= 0.
\end{align*}
So
$\mu(\alpha(x),\mu(y,z))+(-1)^{|x||y|}\mu(\alpha(y),\mu(x,z))=0.$
Hence
$\mu(\alpha(x),\mu(y,z))=-(-1)^{|x||y|}\mu(\alpha(y),\mu(x,z)).$
\end{proof}
\subsection{Construction Theorems and Examples}
In this section, we prove that the category of Hom-alternative superalgebras is closed under self weak morphism. This procedure was applied to associative $($super-$)$ algebras, $G$-associative $($super-$)$ algebras and Lie $($super-$)$ algebras in \cite{ammar2010hom, makhloufand silvestrov2006hom}.
It was introduced first in \cite[Theorem (2.3)]{Yau:homology} and this procedure is called twisting principle or construction by composition.
\begin{df} Let $(\mathcal{A},\mu)$ be a given superalgebra and $\alpha: \mathcal{A}\longrightarrow \mathcal{A}$ be an even superalgebra morphism. Define the Hom-superalgebra induced by $\alpha$ as
$$\mathcal{A}_{\alpha}=(\mathcal{A},\mu_{\alpha}=\alpha \circ \mu,\alpha).$$
\end{df}
\begin{thm}\label{induced A} Let $(\mathcal{A},\mu,\alpha)$ be a left Hom-alternative superalgebra $($resp. right Hom-alternative superalgebra$)$
and $\beta: \mathcal{A} \longrightarrow \mathcal{A}$ be an even left alternative superalgebra endomorphism $($resp. right alternative superalgebra$)$. Then $\mathcal{A}_{\beta}=(\mathcal{A},\mu_{\beta}=\beta \circ \mu,\beta\alpha)$ is a left Hom-alternative superalgebra $($resp. right Hom-alternative superalgebra$)$.\\
Moreover, suppose that $(\mathcal{A}^{'},\mu^{'})$ is an other left alternative superalgebra $($resp. right alternative superalgebra$)$ and $\alpha^{'}: \mathcal{A}^{'}\longrightarrow \mathcal{A}^{'}$ be a left alternative superalgebra endomorphism $($resp. right alternative superalgebra endomorphism$)$. If $f:\mathcal{A}\longrightarrow \mathcal{A}^{'}$ is a left alternative superalgebra morphism $($resp. right alternative superalgebra morphism$)$ that satisfies $f \circ \beta= \alpha^{'} \circ f$, then
$$ f:(A,\mu_{\beta}=\beta\circ\mu,\beta\alpha)\longrightarrow (A^{'},\mu^{'}_{\alpha^{'}}=\alpha^{'}\circ\mu^{'},\alpha^{'})$$
is a morphism of left Hom-alternative superalgebras $($resp. right Hom-alternative superalgebras$)$.
\end{thm}
\begin{proof} We show that $\mathcal{A}_{\beta}=(\mathcal{A},\mu_{\beta}=\beta \circ \mu,\beta\alpha)$ satisfies the left Hom-alternative super-identity $(\ref{sa1})$
$($resp. right Hom-alternative super-identity $(\ref{sa2}))$. Indeed
\begin{eqnarray*}
 \widetilde{ as_{\mathcal{A}_{\beta}}}(x,y,z) &=& \mu_{\beta}(\beta\alpha(x),\mu_{\beta}(y,z))-\mu_{\beta}(\mu_{\beta}(x,y),\beta\alpha(z))\\
   &=& \beta\circ\mu(\beta\alpha(x),\beta\circ\mu(y,z))-\beta\circ\mu(\beta\circ\mu(x,y),\beta\alpha(z))\\
   &=& \beta^{2}\circ\Big(\mu(\alpha(x),\mu(y,z))-\mu(\mu(x,y),\alpha(z))\Big)\\
   &=& \beta^{2}\circ\Big(-(-1)^{|x||y|}\mu(\alpha(y),\mu(x,z))-\mu(\mu(y,x),\alpha(z))\Big)\\
   &=& -(-1)^{|x||y|}\Big(\beta^{2}\circ(\mu(\alpha(y),\mu(x,z))-\mu(\mu(y,x),\alpha(z)))\Big)\\
   &=& -(-1)^{|x||y|}\Big(\beta\circ\mu(\beta\alpha(y),\beta\circ\mu(x,z))-\beta\circ\mu(\beta\circ\mu(y,x),\beta\alpha(z))\Big)\\
   &=&-(-1)^{|x||y|}\Big(\mu_{\beta}(\beta\alpha(y),\mu_{\alpha}(x,z))-\mu_{\beta}(\mu_{\beta}(y,x),\beta\alpha(z))\Big)\\
   &=&-(-1)^{|x||y|}\widetilde{as_{\mathcal{A}_{\beta}}}(y,x,z).
\end{eqnarray*}
The second assertion follows from
$$f\circ\mu_{\beta}=f\circ\beta\circ\mu=\alpha^{'}\circ f\circ\mu=\alpha^{'}\circ\mu^{'}\circ f \otimes f=
\mu^{'}_{\alpha^{'}}\circ f \otimes f.$$
\end{proof}
As a particular case we obtain the following Example.
\begin{exa} Let $(\mathcal{A},\mu)$ be an alternative superalgebra and $\alpha$ be an even alternative superalgebra morphism, then $\mathcal{A}_{\alpha}=(\mathcal{A},\mu_{\alpha},\alpha)$ is a multiplicative Hom-alternative superalgebra.
\end{exa}
\begin{rem} Let $(\mathcal{A},\mu,\alpha)$ be a Hom-alternative superalgebra, one may ask whether this Hom-alternative superalgebra is induced by an ordinary alternative superalgebra $(\mathcal{A},\widetilde{\mu})$, that is $\alpha$ is an even superalgebra endomorphism with respect to $\widetilde{\mu}$ and $\mu=\alpha \circ \widetilde{\mu}$. This question was adressed and discussed for Hom-associative algebras in \cite{Gohr2009hom}.\\
First observation, if $\alpha$ is an even superalgebra endomorphism with respect to $\widetilde{\mu}$. Then $\alpha$ is also an even superalgebra endomorphism with respect to $\mu$. Indeed
\begin{eqnarray*}
\mu(\alpha(x),\alpha(y))&=& \alpha \circ \widetilde{\mu}(\alpha(x),\alpha(y))\\
&=&\alpha \circ \alpha \circ \widetilde{\mu}(x,y)\\
&=&\alpha \circ \mu(x,y).
\end{eqnarray*}
Second observation, if $\alpha$ is bijective then $\alpha^{-1}$ is also an even superalgebra automorphism. Therefore one may use an untwist operation on the Hom-alternative superalgebra in order to recover the alternative superalgebra $(\widetilde{\mu}=\alpha^{-1}\circ \mu)$.
\end{rem}
\begin{df}  Let $(\mathcal{A},\mu,\alpha)$ be a Hom-superalgebra and $n \geq 0$. Define the \textsf{nth derived} Hom-superalgebra of $\mathcal{A}$ by
$$\mathcal{A}^{n}=(\mathcal{A},\mu^{(n)}=\alpha^{2^{n}-1} \circ \mu,\alpha^{2^{n}}).$$
Note that $\mathcal{A}^{0}=\mathcal{A},~~\mathcal{A}^{1}=(\mathcal{A},\mu^{(1)}=\alpha \circ \mu,\alpha^{2})$, and $\mathcal{A}^{n+1}=(\mathcal{A}^{n})^{1}$.
\end{df}
\begin{cor}\label{derived A} Let $(\mathcal{A},\mu,\alpha)$ be a multiplicative Hom-alternative superalgebra. Then the \textsf{nth} derived Hom-superalgebra $\mathcal{A}^{n}=(\mathcal{A},\mu^{(n)}=\alpha^{2^{n}-1} \circ \mu,\alpha^{2^{n}})$ is also a multiplicative Hom-alternative superalgebra for each $n \geq 0$.
\end{cor}
\section{Hom-Malcev Superalgebras}
We introduce and study in this section Hom-Malcev superalgebras. Other characterizations of the Hom-Malcev super-identity are given by Proposition $(\ref{carac hom-malcev})$. We provide some construction results for Hom-Malcev superalgebras
Theorem $(\ref{derivMALCEV})$ and Theorem $(\ref{inducedMALCEV})$. Then using Theorem $(\ref{inducedMALCEV})$, we construct $($non-Hom-Lie$)$ Hom-Malcev superalgebra $($Example $(\ref{exempl-Malcev1}))$.
\begin{df}\label{hom-lie}\begin{enumerate}\ \item A \textsf{Hom-Lie superalgebra} \cite{ammar2010hom} is a Hom-superalgebra $(\mathcal{A},[-,-],\alpha)$ such that $[-,-]$ is super skewsymmetric $($i.e. $[x,y]=-(-1)^{|x||y|}[y,x])$ and that the \textsf{Hom-Jacobian super-identity}
\begin{equation}\label{HSJACOBIAN}
    \widetilde{J}_{\mathcal{A}}(x,y,z)=[[x,y],\alpha(z)]-[\alpha(x),[y,z]]-(-1)^{|y||z|}[[x,z],\alpha(y)]=0
\end{equation}
is satisfied for all $x,y,z$ in $\mathcal{H}(\mathcal{A}).$
\item A \textsf{Hom-Malcev superalgebra} is a Hom-superalgebra $(\mathcal{A},[-,-],\alpha)$ such that $[-,-]$ is super skewsymmetric
$($i.e. $[x,y]=-(-1)^{|x||y|}[y,x])$ and that the \textsf{Hom-Malcev super-identity}
\begin{eqnarray}\label{hom-M-super}
2[\alpha^{2}(t),\widetilde{J}_{\mathcal{A}}(x,y,z)]&=&\widetilde{J}_{\mathcal{A}}(\alpha(t),\alpha(x),[y,z])
+(-1)^{|x|(|y|+|z|)}\widetilde{J}_{\mathcal{A}}(\alpha(t),\alpha(y),[z,x])\nonumber\\
&+&(-1)^{|z|(|x|+|y|)}\widetilde{J}_{\mathcal{A}}(\alpha(t),\alpha(z),[x,y])
\end{eqnarray}
is satisfied for all $x,y,z,t \in \mathcal{H}(\mathcal{A})$.
\end{enumerate}
\end{df}
Observe that when $\alpha=Id$, the Hom-Jacobi super-identity reduces to the usual Jacobi super-identity
$$\mathcal{J}_{\mathcal{A}}(x,y,z)=[[x,y],z]-[x,[y,z]]-(-1)^{|y||z|}[[x,z],y]=0$$
for all $x,y,z \in \mathcal{H}(\mathcal{A})$. Likewise, when $\alpha=Id$, by the super anti-symmetry of $[-,-]$, the Hom-Malcev super-identity reduces
to the Malcev super-identity $(\ref{MALCEV-IDI})$ or equivalently
\begin{eqnarray}\label{MALCEV-IDI}
2[t,\mathcal{J}_{\mathcal{A}}(x,y,z)]&=&\mathcal{J}_{\mathcal{A}}(t,x,[y,z])
+(-1)^{|x|(|y|+|z|)}\mathcal{J}_{\mathcal{A}}(t,y,[z,x])\nonumber\\
&+&(-1)^{|z|(|x|+|y|)}\mathcal{J}_{\mathcal{A}}(t,z,[x,y])
\end{eqnarray}
for all $x,y,z$ and $t$ in $\mathcal{H}(\mathcal{A})$.
\begin{exa} A Lie $($resp. Malcev \cite{shest-irre-non-lie-module,shest-with-trivial}) superalgebra $(\mathcal{A},[-,-])$ is a Hom-Lie $($resp. Hom-Malcev$)$ superalgebra with $\alpha=Id$, since
the Hom-Jacobi super-identity (\ref{hjsuper}) $($resp.  Hom-Malcev super-identity (\ref{hom-M-super})) reduces to the usual Jacobi $($resp. Malcev$)$ super-identity. Moreover, every Hom-Lie superalgebra is a Hom-Malcev superalgebra because the Hom-Jacobi super-identity (\ref{HSJACOBIAN}) clearly implies the  Hom-Malcev super-identity $(\ref{hom-M-super})$.
\end{exa}
\begin{lem} If $(\mathcal{A},[-,-],\alpha)$ is a Hom-Malcev superalgebra. Then for all $x,y,z \in \mathcal{H}(\mathcal{A})$
\begin{enumerate}
\item $\widetilde{J}_{\mathcal{A}}(x,y,z)=-(-1)^{|x||y|}\widetilde{J}_{\mathcal{A}}(y,x,z)$,
\item $\widetilde{J}_{\mathcal{A}}(x,y,z)=-(-1)^{|y||z|}\widetilde{J}_{\mathcal{A}}(x,z,y)$,
\item $\widetilde{J}_{\mathcal{A}}(x,y,z)=-(-1)^{|x||y|+|z|(|x|+|y|)}\widetilde{J}_{\mathcal{A}}(z,y,x)$.
\end{enumerate}
\end{lem}
\begin{proof}  Straightforward calculations.
\end{proof}
Now, we aim to provide   examples of Hom-Malcev superalgebras, using twisting principle. Let us give some characterizations of the Hom-Malcev super-identity.
\begin{prop}\label{carac hom-malcev} Let $(\mathcal{A},[-,-],\alpha)$ be a multiplicative Hom-superalgebra where $[-,-]$ is super skewsymmetric. The following statements are equivalent
\begin{enumerate}
\item ~~ $(\mathcal{A},[-,-],\alpha)$ is a Hom-Malcev superalgebra, i.e. the Hom-Malcev super-identity $(\ref{hom-M-super})$ holds.
\item~~ The equality
\begin{eqnarray}\label{hom-M-super2}
&&\widetilde{J}_{\mathcal{A}}(\alpha(x),\alpha(y),[t,z])+(-1)^{|x||y|+|t|(|x|+|y|)}\widetilde{J}_{\mathcal{A}}(\alpha(t),\alpha(y),[x,z])\nonumber\\
&&=(-1)^{|t||z|}[\widetilde{J}_{\mathcal{A}}(x,y,z),\alpha^{2}(t)]+(-1)^{|x|(|y|+|z|+|t|)+|t||y|}[\widetilde{J}_{\mathcal{A}}(t,y,z),\alpha^{2}(x)]
\end{eqnarray}
holds for all $x,y,z,t \in \mathcal{H}(\mathcal{A})$.\\
\item ~~The equality
\begin{eqnarray}\label{hom-M-super3}
&&(-1)^{|x||y|+|t|(|x|+|y|)}\alpha([[t,y],[x,z]])+\alpha([[x,y],[t,z]])\nonumber\\
&&=(-1)^{|t||z|+|x|(|t|+|y|+|z|)}[[[y,z],\alpha(t)],\alpha^{2}(x)]+(-1)^{|t||z|+|x|(|y|+|z|)}[[[y,z],\alpha(x)],\alpha^{2}(t)]\nonumber\\
&&+(-1)^{|z|(|x|+|t|)+|y|(|z|+|t|)}[[[z,x],\alpha(t)],\alpha^{2}(y)]+(-1)^{|t||z|+|y|(|t|+|z|)+|x|(|t|+|z|)}[[[z,t],\alpha(x)],\alpha^{2}(y)]\nonumber\\
&&+(-1)^{|t||y|+|x|(|t|+|y|+|z|)}[[[t,y],\alpha(z)],\alpha^{2}(x)]+(-1)^{|z||t|}[[[x,y],\alpha(z)],\alpha^{2}(t)]
\end{eqnarray}
holds for all $x,y,z$ and $t$ in $\mathcal{H}(\mathcal{A})$.
\end{enumerate}
\end{prop}
\begin{proof} To prove the equivalence between $(\ref{hom-M-super2})$ and $(\ref{hom-M-super3})$, observe that the left hand side of the identity $(\ref{hom-M-super2})$ is
\begin{eqnarray*}
&&\widetilde{J}_{\mathcal{A}}(\alpha(x),\alpha(y),[t,z])+(-1)^{|x||y|+|t|(|x|+|y|)}\widetilde{J}_{\mathcal{A}}(\alpha(t),\alpha(y),[x,z])\\
&&=\alpha([[x,y],[t,z]])-[\alpha^{2}(x),[\alpha(y),[t,z]]]-(-1)^{|y|(|t|+|z|)}[[\alpha(x),[t,z]],\alpha^{2}(y)]\\
&&+(-1)^{|x||y|+|t|(|x|+|y|)}\alpha([[t,y],[x,z]])-(-1)^{|x||y|+|t|(|x|+|y|)}[\alpha^{2}(t),[\alpha(y),[x,z]]]\\
&&-(-1)^{|x||y|+|t|(|x|+|y|)}[\alpha(x),[t,z]],\alpha^{2}(y)].
\end{eqnarray*}
In the last equality above, we use the multiplicativity of $\alpha$ and the super anti-symmetry of $[-,-]$. Likewise, the right-hand side of the identity
$(\ref{hom-M-super2})$ is
\begin{eqnarray*}
&&(-1)^{|t||z|}[\widetilde{J}_{\mathcal{A}}(x,y,z),\alpha^{2}(t)]+(-1)^{|x|(|y|+|z|+|t|)+|t||y|}[\widetilde{J}_{\mathcal{A}}(t,y,z),\alpha^{2}(x)]\\
&&=(-1)^{|t||x|}[[[x,y],\alpha(z)],\alpha^{2}(t)]+(-1)^{|x|(|y|+|z|+|t|)+|t||y|}[[[t,y],\alpha(z)],\alpha^{2}(x)]\\
&&-(-1)^{|z|(|t|+|y|)}[[[x,z],\alpha(y)],\alpha^{2}(t)]-(-1)^{|t||x|}[[\alpha(x),[y,z]],\alpha^{2}(t)]\\
&&-(-1)^{|x|(|y|+|z|+|t|)+|t||y|}[[\alpha(x),[y,z]],\alpha^{2}(t)]\\
&&-(-1)^{|x|(|y|+|z|+|t|)+|y|(|t|+|z|)}[[[t,z],\alpha(y)],\alpha^{2}(x)].
\end{eqnarray*}
Since the two summands $\Big(-[\alpha^{2}(x),[\alpha(y),[t,z]]]\Big) $ and $\Big(-(-1)^{|x||y|+|t|(|x|+|y|)}[\alpha^{2}(t),[\alpha(y),[x,z]]]\Big)$ appears on both sides of $(\ref{hom-M-super2})$, the above calculation and a rearrangement of terms imply the equivalence between $(\ref{hom-M-super2})$ and $(\ref{hom-M-super3}).$
\end{proof}
To state our next result, we need the following Lemma.
\begin{lem} Let $(\mathcal{A},[-,-],\alpha)$ be a multiplicative Hom-superalgebra. Then we have
\begin{equation}\label{car hom-supJaco}
    \widetilde{J}_{\mathcal{A}}\circ \alpha^{\otimes3} =\alpha \circ \widetilde{J}_{\mathcal{A}}
\end{equation}
and
\begin{equation}\label{car hom-deri}
    \widetilde{J}_{\mathcal{A}^{n}}=\alpha^{2(2^{n}-1)} \circ \widetilde{J}_{\mathcal{A}}
\end{equation}
for all $n \geq 0$.
\end{lem}
\begin{proof}Let $x,y,z \in \mathcal{H}(\mathcal{A})$. We have
\begin{eqnarray*}
  \widetilde{J}_{\mathcal{A}}(\alpha(x),\alpha(y),\alpha(z))
  &=&[[\alpha(x),\alpha(y)],\alpha^{2}(z)]-[\alpha^{2}(x),[\alpha(y),\alpha(z)]]
  -(-1)^{|y||z|}[[\alpha(x),\alpha(z)],\alpha^{2}(y)]  \\
   &=& \alpha([[x,y],\alpha(z)])-\alpha([\alpha(x),[y,z]])
   -(-1)^{|y||z|}\alpha([[x,z],\alpha(y)]) \hskip0.1cm (\text{by multiplicativity of $\alpha$}) \\
   &=& \alpha \circ \widetilde{J}_{\mathcal{A}}(x,y,z).
\end{eqnarray*}
So condition $(\ref{car hom-supJaco})$ holds. For $(\ref{car hom-deri})$, we have
\begin{eqnarray*}
   \widetilde{J}_{\mathcal{A}^{n}}(x,y,z)
   &=& [[x,y]^{(n)},\alpha^{2^{n}}(z)]^{(n)}-[\alpha^{2^{n}}(x),[y,z]^{(n)}]^{(n)}
   -(-1)^{|y||z|}[[x,z]^{(n)},\alpha^{2^{n}}(y)]^{(n)} \\
   &=& \alpha^{2^{n}-1} \circ [\alpha^{2^{n}-1}([x,y]),\alpha^{2^{n}}(z)]
   -\alpha^{2^{n}-1}\circ [\alpha^{2^{n}}(x),\alpha^{2^{n}-1}([y,z])]\\
   &-&(-1)^{|y||z|}\alpha^{2^{n}-1}\circ [\alpha^{2^{n}-1}([x,z]),\alpha^{2^{n}}(y)] \\
   &=&\alpha^{2(2^{n}-1)}\circ \widetilde{J}_{\mathcal{A}}(x,y,z),
\end{eqnarray*}
where the last equality follows from the multiplicativity of $\alpha$ with respect to the bracket $[-,-]$.
\end{proof}
The following result shows that the category of Hom-Malcev superalgebras is closed under taking derived Hom-superalgebras.

\begin{prop}\label{derivMALCEV} Let $(\mathcal{A},[-,-],\alpha)$ be a multiplicative Hom-Malcev superalgebra. Then the derived Hom-superalgebra $\mathcal{A}^{n}=(\mathcal{A},[-,-]^{(n)}=\alpha^{2^{n}-1} \circ [-,-],\alpha^{2^{n}})$ is also a Hom-Malcev superalgebra for any $n \geq 0$.
\end{prop}
\begin{proof} Since $\mathcal{A}^{0}=\mathcal{A},~\mathcal{A}^{1}=(\mathcal{A},[-,-]^{(1)}=\alpha \circ [-,-],\alpha^{2}),$ and $\mathcal{A}^{n+1}=(\mathcal{A}^{n})^{1}$, by an induction argument it is enough  to prove the case $n=1$.\\
To show that $\mathcal{A}^{1}$ is a Hom-Malcev superalgebra, first note that $[-,-]^{(1)}$ is super skewsymmetric because $[-,-]$ is super skewsymmetric and
$\alpha$ is linear. Since $\alpha^{2}$ is multiplicative with respect to $[-,-]^{(1)}$, it remains to show the Hom-Malcev super-identity for $\mathcal{A}^{1}$. For all $x,y,z$ and $t$ in $\mathcal{H}(\mathcal{A})$, we have
\begin{eqnarray*}
&& \widetilde{J}_{\mathcal{A}^{1}}(\alpha^{2}(x),\alpha^{2}(y),[t,z]^{(1)})
+(-1)^{|x||y|+|t|(|x|+|y|)}\widetilde{J}_{\mathcal{A}^{1}}(\alpha^{2}(t),\alpha^{2}(y),[x,z]^{(1)})\\
&&=\alpha^{2} \circ \Big(\widetilde{J}_{\mathcal{A}}(\alpha^{2}(x),\alpha^{2}(y),\alpha([t,z]))
+(-1)^{|x||y|+|t|(|x|+|y|)}\widetilde{J}_{\mathcal{A}}(\alpha^{2}(t),\alpha^{2}(y),\alpha([x,z]))\Big)\\
&&=\alpha^{3} \circ \Big(\widetilde{J}_{\mathcal{A}}(\alpha(x),\alpha(y),[t,z])
+(-1)^{|x||y|+|t|(|x|+|y|)}\widetilde{J}_{\mathcal{A}}(\alpha(t),\alpha(y),[x,z])\Big)\\
&&=\alpha^{3} \circ \Big((-1)^{|t||z|}[\widetilde{J}_{\mathcal{A}^{1}}(x,y,z),\alpha^{2}(t)]
+(-1)^{|x|(|y|+|z|+|t|)+|t||y|}[\widetilde{J}_{\mathcal{A}^{1}}(t,y,z),\alpha^{2}(x)]\Big)\\
&&=(-1)^{|t||z|}[\alpha^{2} \circ\widetilde{J}_{\mathcal{A}^{1}}(x,y,z),(\alpha^{2})^{2}(t)]^{(1)}
+(-1)^{|x|(|y|+|z|+|t|)+|t||y|}[\alpha^{2} \circ\widetilde{J}_{\mathcal{A}^{1}}(t,y,z),(\alpha^{2})^{2}(x)]^{(1)}\\
&&=(-1)^{|t||z|}[\widetilde{J}_{\mathcal{A}^{1}}(x,y,z),(\alpha^{2})^{2}(t)]^{(1)}
+(-1)^{|x|(|y|+|z|+|t|)+|t||y|}[\widetilde{J}_{\mathcal{A}^{1}}(t,y,z),(\alpha^{2})^{2}(x)]^{(1)}.
\end{eqnarray*}
This establishes the Hom-Malcev super-identity in $\mathcal{A}^{1}$ and finishes the proof.
\end{proof}
The next result shows that given a Hom-Malcev superalgebra and an even alternative superalgebra morphism, the induced Hom-superalgebra is a Hom-Malcev superalgebra.
\begin{thm}\label{inducedMALCEV} Let $(\mathcal{A},[-,-],\alpha)$ be a Hom-Malcev superalgebra and $\beta: \mathcal{A} \longrightarrow \mathcal{A}$ be an even Malcev superalgebra endomorphism. Then $\mathcal{A}_{\beta}=(\mathcal{A},[-,-]_{\beta}=\beta \circ [-,-],\beta\alpha)$ is a Hom-Malcev superalgebra.
\end{thm}
\begin{proof} Since $[-,-]_{\beta}$ is super skewsymmetric, it remains to prove the Hom-Malcev super-identity $(\ref{hom-M-super})$ for $\mathcal{A}_{\beta}$. Here we regard $\mathcal{A}$ as the Hom-Malcev superalgebra $(\mathcal{A},[-,-],Id)$ with identity twisting map. For any superalgebra $(\mathcal{A},[-,-])$, by the multiplicativity of $\beta$ with respect to $[-,-],$ we have
$$[-,-]_{\beta}\circ ([-,-]_{\beta}\otimes\beta\alpha)=\beta^{2}\circ [-,-]\circ ([-,-]\otimes \beta)$$
and
$$[-,-]\circ ([-,-]\otimes Id)\circ \beta^{\otimes3} =\beta \circ [-,-]\circ ([-,-]\otimes Id).$$
Pre-composing these identities with the cyclic sum $(Id+\sigma+\sigma^{2})(\ref{Hom-super-Jac})$ (which commutes with $\beta^{\otimes3}$), we obtain
\begin{equation}\label{Jacobian A alpha}
    \widetilde{J}_{\mathcal{A}_{\beta}}=\beta^{2}\circ \widetilde{J}_{\mathcal{A}}
\end{equation}
and
\begin{equation}\label{Jacobian rong alpha}
    \widetilde{J}_{\mathcal{A}}\circ \beta^{\otimes3} =\beta \circ \widetilde{J}_{\mathcal{A}}.
\end{equation}
To prove the Hom-Malcev super-identity for $\mathcal{A}_{\beta}$, we compute
\begin{eqnarray*}
 &&\widetilde{J}_{\mathcal{A}_{\beta}}(\beta\alpha(x),\beta\alpha(y),[t,z]_{\beta})
 +(-1)^{|x||y|+|t|(|x|+|y|)}\widetilde{J}_{\mathcal{A}_{\beta}}(\beta\alpha(t),\beta\alpha(y),[x,z]_{\beta}) \\
 &&=\beta^{2}\circ \Big(\widetilde{J}_{\mathcal{A}}(\beta\alpha(x),\beta\alpha(y),\beta([t,z]))
 +(-1)^{|x||y|+|t|(|x|+|y|)}\widetilde{J}_{\mathcal{A}}(\beta\alpha(t),\beta\alpha(y),\beta([x,z]))\Big)  \\
 &&=\beta^{3}\circ \Big(\widetilde{J}_{\mathcal{A}}(\alpha(x),\alpha(y),[t,z])
 +(-1)^{|x||y|+|t|(|x|+|y|)}\widetilde{J}_{\mathcal{A}}(\alpha(t),\alpha(y),[x,z])\Big)  \\
 &&= \beta^{3}\circ \Big((-1)^{|t||z|}[\widetilde{J}_{\mathcal{A}}(x,y,z),\alpha^{2}(t)]
 +(-1)^{|x|(|y|+|z|+|t|)+|t||y|}[\widetilde{J}_{\mathcal{A}}(t,y,z),\alpha^{2}(x)]\Big)\\
 &&= (-1)^{|t||z|}[\beta^{2}\circ\widetilde{J}_{\mathcal{A}}(x,y,z),(\beta\alpha)^{2}(t)]_{\beta}
 +(-1)^{|x|(|y|+|z|+|t|)+|t||y|}[\beta^{2}\circ\widetilde{J}_{\mathcal{A}}(t,y,z)),(\beta\alpha)^{2}(x)]_{\beta}\\
 &&= (-1)^{|t||z|}[\widetilde{J}_{\mathcal{A}_{\beta}}(x,y,z),(\beta\alpha)^{2}(t)]_{\beta}
 +(-1)^{|x|(|y|+|z|+|t|)+|t||y|}[\widetilde{J}_{\mathcal{A}_{\beta}}(t,y,z),(\beta\alpha)^{2}(x)]_{\beta}.
\end{eqnarray*}
This shows that the Hom-Malcev super-identity holds in $\mathcal{A}_{\beta}$.
\end{proof}
\begin{exa}\label{exempl-Malcev1} $($Hom-Malcev superalgebra of dimension $4)$. Let $M^{3}(3,1)$  be a non-Lie Malcev superalgebra defined with respect to a basis $\{e_1,e_2,e_3,e_4\}$, where $(M^{3}(3,1))_0=span\{e_1,e_2,e_3\}$ and $(M^{3}(3,1))_1=span\{e_4\}$, by the following  multiplication table (see \cite{helena-classification-Malcev-super}) $:$
$$\begin{tabular}{|l|l|l|l|l|}
  \hline
  $[-,-]$ & $e_{1}$  & $e_{2}$ &$e_{3}$ & $e_{4}$ \\ \hline
  $e_{1}$ & $0$ & $0$ & $-e_{1}$ & $0$ \\ \hline
  $e_{2}$ & $0$ & $0$ & $2e_{2}$ & $0$ \\ \hline
  $e_{3}$ & $e_{1}$ & $-2e_{2}$ & $0$ & $-e_{4}$ \\ \hline
  $e_{4}$ & $0$ & $0$ & $e_{4}$ & $e_{1}+e_{2}$ \\
  \hline
\end{tabular}$$
The even superalgebra endomorphism $\alpha_{1}$ with respect to the same  basis  is defined by the matrix
$$\alpha_{1}=\left(
  \begin{array}{cccc}
    a^{2} & 0 & b & 0 \\
    0 & a^{2} & c & 0 \\
    0 & 0 & 1 & 0 \\
    0 & 0 & 0 & a \\
  \end{array}
\right)$$
where $a,b,c \in \mathbb{K}$. For each such even superalgebra morphism $\alpha_{1}:M^{3}(3,1)\longrightarrow M^{3}(3,1)$, by Theorem $(\ref{inducedMALCEV})$
there is a Hom-Malcev superalgebra $M^{3}(3,1)_{\alpha_{1}}=(M^{3}(3,1),[-,-]_{\alpha_{1}},\alpha_{1})$ whose multiplication table is $:$
$$\begin{tabular}{|l|l|l|l|l|}
  \hline
  $[-,-]_{\alpha_{1}}$ & $e_{1}$ & $e_{2}$ & $e_{3}$ & $e_{4}$ \\ \hline
  $e_{1}$ & $0$ & $0$ & $-a^{2}e_{1}$ & $0$ \\ \hline
  $e_{2}$ & $0$ & $0$ & $2a^{2}e_{2}$ & $0$ \\ \hline
  $e_{3}$ & $a^{2}e_{1}$ & $-2a^{2}e_{2}$ & $0$ & $-ae_{4}$ \\ \hline
 $ e_{4}$ & $0$ & $0$ & $ae_{4}$ & $a^{2}(e_{1}+e_{2})$ \\
  \hline
\end{tabular}$$
Note that $(M^{3}(3,1),[-,-]_{\alpha_{1}})$ is in general not a Hom-Lie superalgebra. Indeed, we have
$$\widetilde{J}(e_{3},e_{4},e_{4})=a^{4}(e_{1}-2e_{2}).$$
Then $\widetilde{J}(e_{3},e_{4},e_{4})\neq 0$ whenever $a\neq 0$. So $(M^{3}(3,1),[-,-]_{\alpha_{1}})$ is not a Hom-Lie superalgebra.\\
Also, in general $(M^{3}(3,1),[-,-]_{\alpha_{1}})$ is not a Malcev superalgebra. Indeed, on the one hand we have
\begin{eqnarray*}
 \mathcal{J}(e_{3},e_{4},[e_{4},e_{3}]_{\alpha_{1}})-\mathcal{J}(e_{4},e_{4},[e_{3},e_{3}]_{\alpha_{1}})  &=& a^{5}(2e_{2}-e_{1})
\end{eqnarray*}
which is not equal to $0$ whenever $a\neq 0$. On the other hand, we have
\begin{eqnarray*}
  [\mathcal{J}(e_{3},e_{4},e_{3}),e_{4}]_{\alpha_{1}}-[\mathcal{J}(e_{4},e_{4},e_{3}),e_{4}]_{\alpha_{1}}
  &=&-(a^{6}+a^{5})e_{1}+(2a^{5}-4a^{6})e_{2}-a^{3}e_{4},
\end{eqnarray*}
which is not equal to $0$ whenever $a\neq 0$. Then
$$\mathcal{J}(e_{3},e_{4},[e_{4},e_{3}]_{\alpha_{1}})-\mathcal{J}(e_{4},e_{4},[e_{3},e_{3}]_{\alpha_{1}})\neq [\mathcal{J}(e_{3},e_{4},e_{3}),e_{4}]_{\alpha_{1}}-[\mathcal{J}(e_{4},e_{4},e_{3}),e_{4}]_{\alpha_{1}}.$$
So $(M^{3}(3,1),[-,-]_{\alpha_{1}})$ is not a Malcev superalgebra.\\
Now we provide a twisting of $M^{3}(3,1)$ into a Hom-Lie superalgebra. For example, consider the  class of even superalgebra morphisms $\alpha_{2}:M^{3}(3,1)\longrightarrow M^{3}(3,1)$ given by the matrix
$$\alpha_{2}=\left(
  \begin{array}{cccc}
    0 & 0 & b & 0 \\
    0 & 0 & c & d \\
    0 & 0 & \frac{1}{2} & 0 \\
    0 & 0 & 0 & 0 \\
  \end{array}
\right)$$
where $b,c,d \in \mathbb{K}$. By Theorem $(\ref{inducedMALCEV})$, $M^{3}(3,1)_{\alpha_{2}}=(M^{3}(3,1),[-,-]_{\alpha_{2}},\alpha_{2})$ is a  Hom-Malcev superalgebra with multiplication table $:$
$$\begin{tabular}{|l|l|l|l|l|}
  \hline
  $[-,-]_{\alpha_{2}}$ & $e_{1}$ & $e_{2}$ & $e_{3}$ & $e_{4}$ \\ \hline
  $e_{1}$ & $0$ & $0$ & $0$ & $0$ \\ \hline
  $e_{2}$ & $0$ & $0$ & $0$ & $0$ \\ \hline
  $e_{3}$ & $0$ & $0$ & $0$ & $-de_{2}$ \\ \hline
 $ e_{4}$ & $0$ & $0$ & $de_{2}$ & $0$ \\
  \hline
\end{tabular}$$
It turns out  that $\widetilde{J}_{M^{3}(3,1)_{\alpha_{2}}}=0$, which implies that $M^{3}(3,1)_{\alpha_{2}}$ is a Hom-Lie superalgebra and  clearly  $(M^{3}(3,1),[-,-]_{\alpha_{2}})$ is  a Lie superalgebra.
\end{exa}
\begin{center}\section{Hom-Alternative Superalgebras are Hom-Malcev-admissible}\end{center}
The main purpose of this section is to show that every Hom-alternative superalgebra gives rise to a Hom-Malcev superalgebra via the super-commutator bracket $($Theorem $(\ref{hAandMalcev admissible}))$. This means that Hom-alternative superalgebras are all Hom-Malcev-admissible superalgebras, generalizing the well-known fact that alternative superalgebras are Malcev-admissible. At the end of this section, we consider a $6$-dimensional, $($non-Hom-Lie$)$ Hom-Malcev superalgebra arising from the alternative superalgebra $($Example $(\ref{exp-B(4,2)})).$
\begin{df}\label{supercommutator} Let $(\mathcal{A},\mu,\alpha)$ be a Hom-superalgebra. Define its \textsf{super-commutator} Hom-superalgebra as the Hom-superalgebra
$$\mathcal{A}^{-}=(\mathcal{A},[-,-],\alpha)$$
where $[x,y]=\mu(x,y)-(-1)^{|x||y|}\mu(y,x)$ for all $x,y \in \mathcal{H}(\mathcal{A})$.\\
The multiplication $[-,-]$ is called the \textsf{super-commutator bracket} of $\mu$. We call a Hom-superalgebra $\mathcal{A}$ \textsf{Hom-Malcev-admissible} $($resp. \textsf{Hom-Lie-admissible} \cite{ammar2010hom}) if $\mathcal{A}^{-}$ is a Hom-Malcev $($resp. Hom-Lie$)$ superalgebra $($Definition (\ref{hom-lie})).
\end{df}
It is proved in \cite{ammar2010hom} that, given a Hom-associative superalgebra $\mathcal{A}$, its super-commutator Hom-superalgebra $\mathcal{A}^{-}$ is a Hom-Lie superalgebra. Also, it is known that  the super-commutator superalgebra of any alternative superalgebra is a Malcev-admissible superalgebra. The following main result of this section generalizes both of these facts. It gives us a large class of Hom-Malcev-admissible superalgebras that are in general not Hom-Lie admissible.
\begin{thm}\label{hAandMalcev admissible} Every Hom-alternative superalgebra is Hom-Malcev-admissible.
\end{thm}
The proof of Theorem $(\ref{hAandMalcev admissible})$ depends on the Hom-type analogues of some identities on alternative superalgebras, see for classical case  \cite{helena-super-jor-alte-malcev-with,TRU-rep-alter}. We will first establish some identities about the Hom-associator and the Hom-super-Jacobian. Then we will go back to the proof of Theorem (\ref{hAandMalcev admissible}). In what follows, we often write $\mu(x,y)$ as $xy$ and omit the subscript in the Hom-associator $\widetilde{as_{\mathcal{A}}}$ $(\ref{has})$ when there is no danger of confusion.\\
The following result is a sort of cocycle condition for the Hom-associator of a Hom-alternative superalgebra.
\begin{lem}\label{hom-prop} Let $(\mathcal{A},\mu,\alpha)$ be a Hom-alternative superalgebra. Then the identity
\begin{eqnarray}\label{hom-prop-eg2}
  &&\widetilde{as_{\mathcal{A}}}(t x,\alpha(y),\alpha(z))-(-1)^{|t|(|x|+|y|+|z|)}\widetilde{as_{\mathcal{A}}}(x y,\alpha(z),\alpha(t))
 +(-1)^{(|t|+|x|)(|y|+|z|)}\widetilde{as_{\mathcal{A}}}(y z,\alpha(t),\alpha(x))\nonumber\\
 &&\ \ \ \ \  =\alpha^{2}(t)\widetilde{as_{\mathcal{A}}}(x,y,z)+\widetilde{as_{\mathcal{A}}}(t,x,y)\alpha^{2}(z)
  \end{eqnarray}
  holds for all $x,y,z$ and $t$ in $\mathcal{H}(\mathcal{A})$.\\
  The identity (\ref{hom-prop-eg2}) is said to be the Hom-Teichm\"{u}ller super-identity (see  \cite{klein-alter}).
\end{lem}
\begin{proof} Let $x,y,z$ and $t$ in $\mathcal{H}(\mathcal{A})$, we have
\begin{eqnarray*}
&&\widetilde{as_{\mathcal{A}}}(t x,\alpha(y),\alpha(z))-(-1)^{|t|(|x|+|y|+|z|)}\widetilde{as_{\mathcal{A}}}(x y,\alpha(z),\alpha(t))
+(-1)^{(|t|+|x|)(|y|+|z|)}\widetilde{as_{\mathcal{A}}}(y z,\alpha(t),\alpha(x))\\
&&= \widetilde{as_{\mathcal{A}}}(t x,\alpha(y),\alpha(z))-\widetilde{as_{\mathcal{A}}}(\alpha(t),x y, \alpha(z))
+\widetilde{as_{\mathcal{A}}}(\alpha(t),\alpha(x),y z)\\
&&=((t x)\alpha(y))\alpha^{2}(z)-\alpha(t x)(\alpha(y)\alpha(z))-(\alpha(t)(x y))\alpha^{2}(z)
\\  && +\alpha^{2}(t)((x y)\alpha(z))+(\alpha(t) \alpha(x))\alpha(y z)-\alpha^{2}(t)(\alpha(x)(y z))\\
&&=\Big((t x)\alpha(y)-\alpha(t)(x y)\Big)\alpha^{2}(z)+\alpha^{2}(t)\Big((x y)\alpha(z)-\alpha(x)(y z) \Big)
+(\alpha(t) \alpha(x))\alpha(y z)-\alpha(t x)(\alpha(y)\alpha(z))\\
&&=\alpha^{2}(t)\widetilde{as_{\mathcal{A}}}(x,y,z)+\widetilde{as_{\mathcal{A}}}(t,x,y)\alpha^{2}(z).
\end{eqnarray*}
In the third equality above, we used the multiplicativity of $\alpha$ twice. \end{proof}
In a Hom-alternative superalgebra, the Hom-associator is an super-alternating map on three variables. We now build a map
on four variables using the Hom-associator that is super-alternating in a Hom-alternative superalgebra.
\begin{df}  Let $(\mathcal{A},\mu,\alpha)$ be a Hom-superalgebra. Define the $\mathbb{Z}_{2}-$graded Hom-Bruck-Kleinfled function $f:\mathcal{A}\times\mathcal{A}\times\mathcal{A}\times\mathcal{A}\longrightarrow\mathcal{A}$ as the even multi-linear map
   \begin{eqnarray}\label{foncf}
  f(t,x,y,z) &=& \widetilde{as_{\mathcal{A}}}(t x,\alpha(y),\alpha(z))
   - (-1)^{|t|(|x|+|y|+|z|)}\widetilde{as_{\mathcal{A}}}(x,y,z)\alpha^{2}(t)\nonumber\\
   &&-(-1)^{|t||x|}\alpha^{2}(x)\widetilde{as_{\mathcal{A}}}(t,y,z)
\end{eqnarray}
for all $x,y,z$ and $t$ in $\mathcal{H}(\mathcal{A})$. Define another even multi-linear map $F:\mathcal{A}\times\mathcal{A}\times\mathcal{A}\times\mathcal{A}\longrightarrow\mathcal{A}$
\begin{equation}\label{FONCTION}
    F=[-,-]\circ (\alpha^{2}\otimes \widetilde{as_{\mathcal{A}}})\circ (Id-\zeta+\zeta^{2}-\zeta^{3})
\end{equation}
where $[-,-]$ is the super-commutator bracket of $\mu$ and $\zeta$ is the cyclic permutation
$$\zeta(t,x,y,z)=(z,t,x,y).$$
In terms of elements, the map $F$ is given by
\begin{eqnarray}\label{exprF}
 F(t,x,y,z)&=&[\alpha^{2}(t),\widetilde{as_{\mathcal{A}}}(x,y,z)]-(-1)^{|z|(|t|+|x|+|y|)}[\alpha^{2}(z),\widetilde{as_{\mathcal{A}}}(t,x,y)]\nonumber\\
    && +(-1)^{(|t|+|x|)(|y|+|z|)}[\alpha^{2}(y),\widetilde{as_{\mathcal{A}}}(z,t,x)]\nonumber\\
 &&-(-1)^{|t|(|x|+|y|+|z|)}[\alpha^{2}(x),\widetilde{as_{\mathcal{A}}}(y,z,t)]
\end{eqnarray}
for all $x,y,z$ and $t$ in $\mathcal{H}(\mathcal{A})$.
\end{df}
\begin{lem}\label{f et F} Given a Hom-alternative superalgebra $(\mathcal{A},\mu,\alpha)$, we have$:$
$$F=f \circ  (Id-\rho+\rho^{2})$$
where  $\rho=\zeta^{3}$ is the cyclic permutation $\rho(t,x,y,z)=(x,y,z,t).$\\
In terms of elements, the even map $F$ is given by
\begin{eqnarray*}
  F(t,x,y,z) &=&f(t,x,y,z)-(-1)^{|t|(|x|+|y|+|z|)}f(x,y,z,t)  \\
   &&+(-1)^{(|t|+|x|)(|z|+|y|)}f(y,z,t,x)
\end{eqnarray*}
for all $x,y,z$ and $t$ in $\mathcal{H}(\mathcal{A})$.
\end{lem}
\begin{proof} By Lemma $(\ref{hom-prop})$ and definition $(\ref{foncf})$, we have
\begin{eqnarray*}
\alpha^{2}(t)\widetilde{as_{\mathcal{A}}}(x,y,z)+\widetilde{as_{\mathcal{A}}}(t,x,y)\alpha^{2}(z)
&=&\widetilde{as_{\mathcal{A}}}(t x,\alpha(y),\alpha(z))\\
&&-(-1)^{|t|(|x|+|y|+|z|)}\widetilde{as_{\mathcal{A}}}(x y,\alpha(z),\alpha(t))\\
&&+(-1)^{(|t|+|x|)(|y|+|z|)}\widetilde{as_{\mathcal{A}}}(y z,\alpha(t),\alpha(x))\\
&=&f(t,x,y,z)+(-1)^{|t|(|x|+|y|+|z|)}\widetilde{as_{\mathcal{A}}}(x,y,z)\alpha^{2}(t)\\
&&+(-1)^{|x||t|}\alpha^{2}(x)\widetilde{as_{\mathcal{A}}}(t,y,z)-(-1)^{|t|(|x|+|y|+|z|)}f(x,y,z,t)\\
&&-(-1)^{|t|(|y|+|z|)+|x|(|y|+|z|)}\widetilde{as_{\mathcal{A}}}(y,z,t)\alpha^{2}(x)\\
&&-(-1)^{|t|(|x|+|y|+|z|)+|x||y|}\alpha^{2}(y)\widetilde{as_{\mathcal{A}}}(x,z,t)\\
&&+(-1)^{(|t|+|x|)(|z|+|y|)}f(y,z,t,x)\\
&&+(-1)^{|z|(|t|+|x|+|y|)}\widetilde{as_{\mathcal{A}}}(z,t,x)\alpha^{2}(y)\\
&&+(-1)^{|t|(|y|+|z|)+|x|(|y|+|z|)+|y||z|}\alpha^{2}(z)\widetilde{as_{\mathcal{A}}}(y,t,x).
\end{eqnarray*}
Since the Hom-associator $\widetilde{as_{\mathcal{A}}}$ is super-alternating, we have$:$
\begin{eqnarray*}
\widetilde{as_{\mathcal{A}}}(y,t,x)&=&(-1)^{|y|(|t|+|x|)}\widetilde{as_{\mathcal{A}}}(y,t,x),\\
\widetilde{as_{\mathcal{A}}}(x,z,t)&=&(-1)^{|z|(|x|+|t|)}\widetilde{as_{\mathcal{A}}}(z,t,x),
\end{eqnarray*}
 and
 \begin{eqnarray*}
 \widetilde{as_{\mathcal{A}}}(t,y,z)&=&(-1)^{|y|(|t|+|z|)}\widetilde{as_{\mathcal{A}}}(y,z,t).
\end{eqnarray*}
\\ Therefore, rearranging terms in the above equality, we have$:$
\begin{eqnarray*}
 [\alpha^{2}(t),\widetilde{as_{\mathcal{A}}}(x,y,z)]&=&\alpha^{2}(t)\widetilde{as_{\mathcal{A}}}(x,y,z)\\
&&-(-1)^{|t|(|x|+|y|+|z|)}\widetilde{as_{\mathcal{A}}}(x,y,z)\alpha^{2}(t),\\
-(-1)^{|z|(|t|+|x|+|y|)}[\alpha^{2}(z),\widetilde{as_{\mathcal{A}}}(t,x,y)]&=&
\widetilde{as_{\mathcal{A}}}(t,x,y)\alpha^{2}(z)\\
&&-(-1)^{|t|(|y|+|z|)+|x|(|y|+|z|)+|y||z|}\alpha^{2}(z)\widetilde{as_{\mathcal{A}}}(y,t,x) ,\\
(-1)^{(|t|+|x|)(|y|+|z|)}[\alpha^{2}(y),\widetilde{as_{\mathcal{A}}}(z,t,x)]&=&
(-1)^{|t|(|x|+|y|+|z|)+|x||y|}\alpha^{2}(y)\widetilde{as_{\mathcal{A}}}(x,z,t)\\
&&-(-1)^{|y|(|x|+|t|+|z|)}\widetilde{as_{\mathcal{A}}}(z,t,x)\alpha^{2}(y),
\end{eqnarray*}
and
\begin{eqnarray*}
\hskip-1cm-(-1)^{|t|(|x|+|y|+|z|)}[\alpha^{2}(x),\widetilde{as_{\mathcal{A}}}(y,z,t)]&=&
-(-1)^{|x||t|}\alpha^{2}(x)\widetilde{as_{\mathcal{A}}}(t,y,z)\\
\hskip-1cm&&+ (-1)^{|t|(|y|+|z|)+|x|(|y|+|z|)}\widetilde{as_{\mathcal{A}}}(y,z,t)\alpha^{2}(x).
\end{eqnarray*}
Then, we obtain $F=f \circ  (Id-\rho+\rho^{2})$ in the explicit form $(\ref{exprF})$.
\end{proof}
\begin{prop} Let $(\mathcal{A},\mu,\alpha)$ be a Hom-alternative superalgebra. Then the $\mathbb{Z}_{2}-$graded Hom-Bruck-Kleinfled function $f$ is super-alternating.
\end{prop}
\begin{proof} Let $x,y,z$ and $t$ in $\mathcal{H}(\mathcal{A})$, by definition $(\ref{exprF})$ we have
\begin{eqnarray*}
&& F(t,x,y,z)+(-1)^{|t|(|z|+|x|+|t|)}F(x,y,z,t)\\
&& =[\alpha^{2}(t),\widetilde{as_{\mathcal{A}}}(x,y,z)]-(-1)^{|z|(|t|+|x|+|y|)}[\alpha^{2}(z),\widetilde{as_{\mathcal{A}}}(t,x,y)]\\
&&+(-1)^{(|t|+|x|)(|y|+|z|)}[\alpha^{2}(y),\widetilde{as_{\mathcal{A}}}(z,t,x)]
-(-1)^{|t|(|x|+|y|+|z|)}[\alpha^{2}(x),\widetilde{as_{\mathcal{A}}}(y,z,t)]\\
&&+(-1)^{|t|(|z|+|x|+|z|)} [\alpha^{2}(x),\widetilde{as_{\mathcal{A}}}(y,z,t)]-[\alpha^{2}(t),\widetilde{as_{\mathcal{A}}}(x,y,z)]\\
&&+(-1)^{|z|(|x|+|y|+|t|)}[\alpha^{2}(z),\widetilde{as_{\mathcal{A}}}(t,x,y)]
-(-1)^{(|t|+|x|)(|y|+|z|)}[\alpha^{2}(y),\widetilde{as_{\mathcal{A}}}(z,t,x)]\\
&&=0.
\end{eqnarray*}
In other hand by Lemma $(\ref{f et F})$, we have$:$
\begin{eqnarray*}
 F(t,x,y,z) &=&f(t,x,y,z)-(-1)^{|t|(|x|+|y|+|z|)}f(x,y,z,t)  \\
   &=& +(-1)^{(|t|+|x|)(|y|+|z|)}f(y,z,t,x).
\end{eqnarray*}
Then
\begin{eqnarray*}
 F(t,x,y,z)+(-1)^{|t|(|z|+|x|+|z|)}F(x,y,z,t) &=&f(t,x,y,z)+(-1)^{(|t|+|x|)(|y|+|z|)}f(z,t,x,y)  \\
   &=& 0,
\end{eqnarray*}
 which implies that $f$ is super-alternating.
 \end{proof}
Next, we provide further properties of the $\mathbb{Z}_{2}-$graded Hom-Bruck-Kleinfeld function $f$ $($Definition $(\ref{foncf}))$. The following result gives two characterizations of the $\mathbb{Z}_{2}-$graded Hom-Bruck-Kleinfeld function in a Hom-alternative superalgebra $(\mathcal{A},\mu,\alpha)$.
\begin{cor} Let $(\mathcal{A},\mu,\alpha)$ be a Hom-alternative superalgebra. Then the $\mathbb{Z}_{2}-$graded Hom-Bruck-Kleinfeld function $f$ satisfies
\begin{eqnarray}\label{cara f}
    f(t,x,y,z)&=&\frac{1}{3}F(t,x,y,z)\nonumber\\
&=&\widetilde{as_{\mathcal{A}}}([t,x],\alpha(y),\alpha(z))
+(-1)^{(|y|+|z|)(|x|+|t|)}\widetilde{as_{\mathcal{A}}}([y,z],\alpha(t),\alpha(x))
\end{eqnarray}
for all $x,y,z$ and $t$ in $\mathcal{H}(\mathcal{A})$, where $[x,y]=xy-(-1)^{|x||y|}yx.$
\end{cor}
\begin{proof} by Lemma $(\ref{f et F})$, we have
$F=f \circ  (Id-\rho+\rho^{2})$
but $f$ is super-alternating, then
$F=3f.$
This  proves the first identity in $(\ref{cara f})$. It remains to prove that $f$ is equal to the last entry in $(\ref{cara f})$.\\
Since $f$ is super-alternating, from its definition $(\ref{foncf})$, we have
\begin{eqnarray*}
  2f(t,x,y,z) &=& f(t,x,y,z)-(-1)^{|x||t|}f(x,t,y,z) \\
   &=& \widetilde{as_{\mathcal{A}}}(t x,\alpha(y),\alpha(z))-(-1)^{(|t|(|x|+|y|+|z|)}\widetilde{as_{\mathcal{A}}}(x,y,z)\alpha^{2}(t) \\
   &&-(-1)^{|x||t|} \alpha^{2}(x)\widetilde{as_{\mathcal{A}}}(t,y,z)-(-1)^{|x||t|}\widetilde{as_{\mathcal{A}}}(x t,\alpha(y),\alpha(z)) \\
   &&+(-1)^{|x|(|y|+|z|)} \widetilde{as_{\mathcal{A}}}(t,y,z)\alpha^{2}(x)+\alpha^{2}(t) \widetilde{as_{\mathcal{A}}}(x,y,z).
   \end{eqnarray*}
Rearranging terms we obtain
\begin{eqnarray}\label{lemm}
   \widetilde{as_{\mathcal{A}}}([t,x],\alpha(y),\alpha(z))
   &=&(-1)^{|x||t|}[\alpha^{2}(x),\widetilde{as_{\mathcal{A}}}(t,y,z)]-[\alpha^{2}(t),\widetilde{as_{\mathcal{A}}}(x,y,z)]+2f(t,x,y,z)\nonumber\\
   &=&(-1)^{(|t|(|x|+|y|+|z|)}[\alpha^{2}(x),\widetilde{as_{\mathcal{A}}}(y,z,t)]\nonumber\\
   &&-[\alpha^{2}(t),\widetilde{as_{\mathcal{A}}}(x,y,z)]+ 2f(t,x,y,z),
  \end{eqnarray}
in which $\widetilde{as_{\mathcal{A}}}(t,y,z)=(-1)^{|t|(|y|+|z|)}\widetilde{as_{\mathcal{A}}}(y,z,t)$ because $\widetilde{as_{\mathcal{A}}}$ is super-alternating. Interchanging $(t,x)$ with $(y,z)$ in $(\ref{lemm})$ and using the super-alternativity of $f$ , we obtain
\begin{eqnarray}\label{lemm2}
 (-1)^{(|t|+|x|)(|y|+|z|)} \widetilde{as_{\mathcal{A}}}([y,z],\alpha(t),\alpha(x))&=&(-1)^{(|y|(|t|+|x|+|z|)}[\alpha^{2}(z),\widetilde{as_{\mathcal{A}}}(t,x,y)]\nonumber\\
&&-(-1)^{(|t|+|x|)(|y|+|z|)} [\alpha^{2}(y),\widetilde{as_{\mathcal{A}}}(z,t,x)]\nonumber\\
 &&+2(-1)^{(|t|+|x|)(|y|+|z|)} f(y,z,t,x)\nonumber\\
&=&(-1)^{(|y|(|t|+|x|+|z|)}[\alpha^{2}(z),\widetilde{as_{\mathcal{A}}}(t,x,y)]\nonumber\\
&&-(-1)^{(|t|+|x|)(|y|+|z|)} [\alpha^{2}(y),\widetilde{as_{\mathcal{A}}}(z,t,x)]\nonumber\\
&&+2f(t,x,y,z).
\end{eqnarray}
Adding $(\ref{lemm})$ and $(\ref{lemm2})$, we have$:$
\begin{eqnarray*}
\widetilde{as_{\mathcal{A}}}([t,x],\alpha(y),\alpha(z))&+&(-1)^{(|t|+|x|)(|y|+|z|)} \widetilde{as_{\mathcal{A}}}([y,z],\alpha(t),\alpha(x))\\
&=&(4f-F)(t,x,y,z).
\end{eqnarray*}
Since $F=3f$. This proves that $f$ is equal to to the last entry in $(\ref{cara f})$.
\end{proof}
The following result is the Hom-type analogue of part of \cite {shest-rep-alter} and says that every Hom-alternative superalgebra satisfies a variation of the Hom-Malcev super-identity $(\ref{hom-M-super})$
in which the Hom-super-Jacobian is replaced by the Hom-associator.
\begin{prop}\label{pop-utilise} Let $(\mathcal{A},\mu,\alpha)$ be a Hom-alternative superalgebra. Then
\begin{eqnarray}\label{2ass-hom}
  2 [\alpha^{2}(t),\widetilde{as_{\mathcal{A}}}(x,y,z)]&=&\widetilde{as_{\mathcal{A}}}(\alpha(t),\alpha(x),[y,z])\nonumber\\
  &&+(-1)^{|x|(|y|+|z|)}\widetilde{as_{\mathcal{A}}}(\alpha(t),\alpha(y),[z,x])\nonumber\\
   &&+(-1)^{|z|(|x|+|y|)}\widetilde{as_{\mathcal{A}}}(\alpha(t),\alpha(z),[x,y])
  \end{eqnarray}
for all $x,y,z$ and $t$ in $\mathcal{H}(\mathcal{A})$, where $[-,-]$ is the super-commutator bracket.
\end{prop}
Next we consider the relationship between the Hom-associator $(\ref{has})$ in a Hom-superalgebra $(\mathcal{A},\mu,\alpha)$ and
the Hom-super-Jacobian $(\ref{hjsuper})$ in its super-commutator Hom-superalgebra $\mathcal{A}^{-}=(\mathcal{A},[-,-],\alpha)$
$($Definition $(\ref{supercommutator}))$.
\begin{lem}\label{superja and ass} Let $(\mathcal{A},\mu,\alpha)$ be a Hom-superalgebra. Then
\begin{eqnarray*}
\widetilde{J}_{\mathcal{A}^{-}}(x,y,z)&=& \widetilde{as_{\mathcal{A}}}(x,y,z)+(-1)^{|x|(|y|+|z|)}\widetilde{as_{\mathcal{A}}}(y,z,x)\\
&&+(-1)^{|z|(|x|+|y|)}\widetilde{as_{\mathcal{A}}}(z,x,y)-(-1)^{|x||y|}\widetilde{as_{\mathcal{A}}}(y,x,z)\\
&&-(-1)^{|y||z|}\widetilde{as_{\mathcal{A}}}(x,z,y)-(-1)^{|x||y|+|z|(|x|+|y|)}\widetilde{as_{\mathcal{A}}}(z,y,x)
\end{eqnarray*}
for all $x,y,z$ in $\mathcal{H}(\mathcal{A})$.
\end{lem}
\begin{proof} Let $(\mathcal{A},\mu,\alpha)$ be a Hom-superalgebra. Then for all $x,y,z$ in $\mathcal{H}(\mathcal{A})$, we have\\
\begin{eqnarray*}&& \widetilde{J}_{\mathcal{A}^{-}}(x,y,z)\\
&&   \ \  =[[x,y],\alpha(z)]-[\alpha(x),[y,z]-(-1)^{|y||z|}[[x,z],\alpha(y)]\\
&&   \ \  =[xy-(-1)^{|x||y|}yx,\alpha(z)]-[\alpha(x),yz-(-1)^{|y||z|}zy]-(-1)^{|y||z|}[xz-(-1)^{|x||z|}zx,\alpha(y)]\\
&&   \ \  =\Big((xy)\alpha(z)-(-1)^{|z|(|x|+|y|)}\alpha(z)(xy)-(-1)^{|x||y|}(yx)\alpha(z)+(-1)^{|x||y|+|z|(|x|+|y|)}\alpha(z)(yx)\Big)\\
&&   \ \  -\Big(\alpha(x)(yz)-(-1)^{|x|(|y|+|z|)}(yz)\alpha(x)-(-1)^{|y||z|}\alpha(x)(zy)+(-1)^{|x|(|y|+|z|)+|y||z|}(zy)\alpha(x)\Big)\\
&&   \ \  -(-1)^{|y||z|}\Big((xz)\alpha(y)-(-1)^{|y|(|x|+|z|)}\alpha(y)(xz)-(-1)^{|x||z|}(zx)\alpha(y)+(-1)^{|x||z|+|y|(|x|+|z|)}\alpha(y)(zx)\Big)\\
&&   \ \  =\widetilde{as_{\mathcal{A}}}(x,y,z)+(-1)^{|x|(|y|+|z|)}\widetilde{as_{\mathcal{A}}}(y,z,x)+(-1)^{|z|(|x|+|y|)}\widetilde{as_{\mathcal{A}}}(z,x,y)\\
&&   \ \  -(-1)^{|x||y|}\widetilde{as_{\mathcal{A}}}(y,x,z)-(-1)^{|y||z|}\widetilde{as_{\mathcal{A}}}(x,z,y)-(-1)^{|x||y|+|z|(|x|+|y|)}\widetilde{as_{\mathcal{A}}}(z,y,x),
\end{eqnarray*}
which completes the proof.
\end{proof}
\begin{prop}\label{6-associ} In a Hom-alternative superalgebra $(\mathcal{A},\mu,\alpha)$, the identity
$$ \widetilde{J}_{\mathcal{A}^{-}}=6\widetilde{as_{\mathcal{A}}}$$
holds.
\end{prop}
\begin{proof} Since the Hom-associator $\widetilde{as_{\mathcal{A}}}$ is super-alternating, with the notations in Lemma $(\ref{superja and ass})$, we have$:$
\begin{eqnarray*}
(-1)^{|x|(|y|+|z|)}\widetilde{as_{\mathcal{A}}}(y,z,x)&=&\widetilde{as_{\mathcal{A}}}(x,y,z),\\
(-1)^{|z|(|x|+|y|)}\widetilde{as_{\mathcal{A}}}(z,x,y)&=&\widetilde{as_{\mathcal{A}}}(x,y,z),\\
-(-1)^{|x||y|}\widetilde{as_{\mathcal{A}}}(y,x,z)&=&\widetilde{as_{\mathcal{A}}}(x,y,z),\\
-(-1)^{|y||z|}\widetilde{as_{\mathcal{A}}}(x,z,y)&=&\widetilde{as_{\mathcal{A}}}(x,y,z),\\
-(-1)^{|x||y|+|z|(|x|+|y|)}\widetilde{as_{\mathcal{A}}}(z,y,x)&=&\widetilde{as_{\mathcal{A}}}(x,y,z).
\end{eqnarray*}
Then the result follows from Lemma $(\ref{superja and ass})$.
\end{proof}
Now, we are ready to prove Theorem \ref{hAandMalcev admissible}.\\
\begin{proof}(Theorem (\ref{hAandMalcev admissible})) Let $(\mathcal{A},\mu,\alpha)$ be a Hom-alternative superalgebra and $\mathcal{A}^{-}=(\mathcal{A},[-,-],\alpha)$ be its super-commutator Hom-superalgebra. The super-commutator bracket $[-,-]$ is super skewsymmetric. Thus, it remains to show that the Hom-Malcev super-identity $(\ref{hom-M-super})$ holds in $\mathcal{A}^{-}$, that is
\begin{eqnarray*}
2[\alpha^{2}(t),\widetilde{J}_{\mathcal{A}^{-}}(x,y,z)]
&=&\widetilde{J}_{\mathcal{A}^{-}}(\alpha(t),\alpha(x),[y,z])
+(-1)^{|x|(|y|+|z|)}\widetilde{J}_{\mathcal{A}^{-}}(\alpha(t),\alpha(y),[z,x])\\
 &&+(-1)^{|z|(|x|+|y|)}\widetilde{J}_{\mathcal{A}^{-}}(\alpha(t),\alpha(z),[x,y]).
\end{eqnarray*}
To prove this, we compute
\begin{eqnarray*}
  2[\alpha^{2}(t),\widetilde{J}_{\mathcal{A}^{-}}(x,y,z)]
  &=& 2[\alpha^{2}(t),6\widetilde{as_{\mathcal{A}}}(x,y,z)]\hskip3.5cm (\text{by Proposition}~\ref{6-associ})  \\
   &=& 6\Big( \widetilde{as_{\mathcal{A}}}(\alpha(t),\alpha(x),[y,z])
   +(-1)^{|x|(|y|+|z|)} \widetilde{as_{\mathcal{A}}}(\alpha(t),\alpha(y),[z,x])\\
   &&+(-1)^{|z|(|x|+|y|)} \widetilde{as_{\mathcal{A}}}(\alpha(t),\alpha(z),[x,y])\Big) \hskip1cm (\text{by Proposition}~\ref{pop-utilise})\\
   &=& \widetilde{J}_{\mathcal{A}^{-}}(\alpha(t),\alpha(x),[y,z])
   +(-1)^{|x|(|y|+|z|)}\widetilde{J}_{\mathcal{A}^{-}}(\alpha(t),\alpha(y),[z,x])\\
   &&+(-1)^{|z|(|x|+|y|)}\widetilde{J}_{\mathcal{A}^{-}}(\alpha(t),\alpha(z),[x,y])\hskip1.3cm (\text{by Proposition}~\ref{6-associ}).
\end{eqnarray*}
\end{proof}
\begin{exa}\label{exp-B(4,2)} In this example, we describe a Hom-alternative superalgebra $($hence a Hom-Malcev-admissible superalgebra by Theorem
$(\ref{hAandMalcev admissible}))$ that is not Hom-Lie-admissible and not alternative. Let $\mathbb{F}$ be a field of characteristic $3$ and let $B(4,2)=B_{0}\oplus B_{1}$ be a $6$-dimensional simple alternative superalgebra (see \cite{shest-rep-alter,shest-Non-commutative}), where $B_{0}=M_{2}(\mathbb{F})$ and $B_{1}=\mathbb{F}\cdot m_{1} +\mathbb{F}\cdot m_{2}$ is the $2$-dimensional irreducible Cayley bi-module over $B_{0}$. Now $B(4,2)$ has basis
$\{e_{11},e_{12},e_{21},e_{22}, m_{1},m_{2}\}$, where the product, with respect to the basis, is given by the following table $:$
$$\begin{tabular}{|l|l|l|l|l|l|l|}
  \hline
  $\mu$ & $e_{11}$ & $e_{12}$ & $e_{21}$ & $e_{22}$ & $m_{1}$ & $m_{2}$ \\ \hline
  $e_{11}$ & $e_{11}$ & $e_{12}$ & $0$ & $0$ & $m_{1}$ & $0$ \\ \hline
  $e_{12}$ & $0$ & $0$ & $e_{11}$ & $e_{12}$ & $m_{2}$ & $0$ \\ \hline
  $e_{21}$ & $e_{21}$ & $e_{22}$ & $0$ & $0$ & $0$ & $m_{1}$ \\ \hline
  $e_{22}$ & $0$ & $0$ & $e_{21}$ & $e_{22}$ & $0$ & $m_{2}$ \\ \hline
  $m_{1}$ & $0$ & $-m_{2}$ & $0$ & $m_{1}$ & $-e_{21}$ & $e_{11}$ \\ \hline
  $m_{2}$ & $m_{2}$ & $0$ & $-m_{1}$ & $0$ & $-e_{22}$ & $e_{12}$ \\
  \hline
\end{tabular}$$
The even superalgebra endomorphisms $\alpha$ with respect to the basis $\{e_{11},e_{12},e_{21},e_{22}, m_{1},m_{2}\}$ are defined by the matrix
$$\alpha=\left(
  \begin{array}{cccccc}
    1 & 0 & \frac{-a}{b} & 0 & 0 & 0 \\
    a & b & -\frac{1}{b} & -a & 0 & 0 \\
    0 & 0 & \frac{1}{b}  & 0 & 0 & 0 \\
    0 & 0 & \frac{a}{b} & 1 & 0 & 0 \\
    0 & 0 & 0 & 0 & \pm\sqrt{\frac{1}{b}} & 0 \\
    0 & 0 & 0 & 0 & \pm a\sqrt{\frac{1}{b}} & \pm b\sqrt{\frac{1}{b}} \\
  \end{array}
\right)$$
where $a,b \neq 0$. According to Theorem $(\ref{induced A})$, the even linear map $\alpha$ an the following multiplication
\begin{eqnarray*}
&& \mu_{\alpha}(e_{11},e_{11})=e_{11}+ae_{12},~~\mu_{\alpha}(e_{11},e_{12})=\mu_{\alpha}(e_{11},e_{22})=\mu_{\alpha}(e_{11},m_{1})=0,
\\ && \mu_{\alpha}(e_{11},e_{21})=\frac{-a}{b}e_{11}-\frac{1}{b}e_{12}+\frac{1}{b}e_{21}+\frac{a}{b}e_{22},~~
\mu_{\alpha}(e_{11},m_{2})=\pm b\sqrt{\frac{1}{b}}m_{2},
\\ && \mu_{\alpha}(e_{12},e_{11})=be_{12},~~\mu_{\alpha}(e_{12},e_{12})=\mu_{\alpha}(e_{12},e_{22})=\mu_{\alpha}(e_{12},m_{2})=0,
\\ && \mu_{\alpha}(e_{12},e_{21})=-ae_{12}+\frac{1}{b}e_{21}+e_{22},~~
\mu_{\alpha}(e_{12},m_{1})=\pm b\sqrt{\frac{1}{b}}m_{2},
\\ && \mu_{\alpha}(e_{21},e_{12})=e_{11}+ae_{12},~~\mu_{\alpha}(e_{21},e_{11})=\mu_{\alpha}(e_{21},e_{21})=\mu_{\alpha}(e_{21},m_{1})=0,
\\ && \mu_{\alpha}(e_{21},e_{22})=\frac{-a}{b}e_{11}-\frac{1}{b}e_{12}+\frac{1}{b}e_{21}+\frac{a}{b}e_{22},~~
\mu_{\alpha}(e_{21},m_{2})=\pm \sqrt{\frac{1}{b}}m_{1}\pm a\sqrt{\frac{1}{b}}m_{2},
\\ &&  \mu_{\alpha}(e_{22},e_{11})=\mu_{\alpha}(e_{22},e_{21})=\mu_{\alpha}(e_{22},m_{2})=0,
\\ && \mu_{\alpha}(e_{22},e_{12})=b e_{12},~~\mu_{\alpha}(e_{22},e_{22})=-a e_{12}+e_{22},~~
\mu_{\alpha}(e_{22},m_{1})=\pm \sqrt{\frac{1}{b}}m_{1}\pm a\sqrt{\frac{1}{b}}m_{2},
\\ && \mu_{\alpha}(m_{1},e_{11})=\pm \sqrt{\frac{1}{b}}m_{1}\pm a\sqrt{\frac{1}{b}}m_{2}, ~~\mu_{\alpha}(m_{1},e_{11})=\mu_{\alpha}(m_{1},e_{21})=0,
\\ && \mu_{\alpha}(m_{1},e_{12})=\pm b\sqrt{\frac{1}{b}}m_{2},~~\mu_{\alpha}(m_{1},m_{1})=\frac{a}{b}e_{11}+\frac{1}{b}e_{12}-\frac{1}{b}e_{21}-\frac{a}{b}e_{22},~~
\mu_{\alpha}(m_{1},m_{2})=ae_{12}-e_{22},
\\ && \mu_{\alpha}(m_{2},e_{21})=\pm \sqrt{\frac{1}{b}}m_{1}\pm a\sqrt{\frac{1}{b}}m_{2}, ~~\mu_{\alpha}(m_{2},e_{11})=\mu_{\alpha}(m_{2},e_{21})=0,
\\ && \mu_{\alpha}(m_{2},e_{22})=\pm b\sqrt{\frac{1}{b}}m_{2},~~\mu_{\alpha}(m_{2},m_{1})=e_{11}+ae_{12},~~
\mu_{\alpha}(m_{2},m_{2})=be_{12},
\end{eqnarray*}
where $a,b \neq 0$. determine a $6$-dimensional Hom-alternative $($hence Hom-Malcev-admissible and Hom-flexible$)$ superalgebra.\\
Note that $B(4,2)_{\alpha}$ is not alternative because
\begin{eqnarray*}
  as_{B(4,2)_{\alpha}}(e_{11},e_{21},e_{22})+as_{B(4,2)_{\alpha}}(e_{21},e_{11},e_{22}) &=& \frac{a}{b}(e_{11}+e_{22}) \\
   &\neq& 0.
   \end{eqnarray*}
Also, $B(4,2)_{\alpha}$ is not Hom-Lie-admissible, that is  $B(4,2)_{\alpha}^{-}$ is not Hom-Lie superalgebra. Indeed, we have
\begin{eqnarray*}
  \widetilde{J}_{B(4,2)_{\alpha}^{-}}(e_{11},e_{21},e_{22}) &=& 6 as_{B(4,2)_{\alpha}}(e_{11},e_{21},e_{22})\\
   &=& 6\frac{a}{b}(e_{11}+e_{22}) \\
   &\neq&  0.
\end{eqnarray*}
Therefore, $B(4,2)_{\alpha}$ is a Hom-alternative $($and hence Hom-Malcev-admissible and Hom-flexible$)$ superalgebra that is neither alternative nor
Hom-Lie-admissible.\\
Also, $(B(4,2),[-,-]_{\alpha})$ is not a Malcev superalgebra. Indeed, on the one hand we have
\begin{eqnarray*}
  [\mathcal{J}(e_{11},e_{12},e_{21}),e_{22}]_{\alpha}+[\mathcal{J}(e_{22},e_{12},e_{21}),e_{11}]_{\alpha} &=&
   (\frac{2a}{b}+1+\frac{1}{b^{2}})e_{11}+(4ab^{2}-2a+\frac{3}{b}+\frac{1}{b^{2}})e_{12} \\
   &+& (\frac{1}{b^{2}}+1-\frac{2}{b^{2}})e_{21}+(1-\frac{2a}{b}+\frac{1}{b^{2}})e_{22},
\end{eqnarray*}
on the other hand, we have
\begin{eqnarray*}
  \mathcal{J}(e_{11},e_{12},[e_{22},e_{21}]_{\alpha})+\mathcal{J}(e_{22},e_{12},[e_{11},e_{21}]_{\alpha}) &=&
  (-\frac{2a}{b^{3}}-2+\frac{2}{b^{2}})e_{11}+(2+\frac{2a}{b^{3}}-\frac{2}{b^{2}})e_{22}\\
   &+& \frac{2}{b^{2}}e_{21}+(2b^{2}+a+\frac{4a}{b^{2}}-\frac{2}{b^{3}}+2ab+\frac{5a}{b})e_{12},
\end{eqnarray*}
where $a,b\neq 0$. Then
$$[\mathcal{J}(e_{11},e_{12},e_{21}),e_{22}]_{\alpha}+[\mathcal{J}(e_{22},e_{12},e_{21}),e_{11}]_{\alpha}\neq
\mathcal{J}(e_{11},e_{12},[e_{22},e_{21}]_{\alpha})+\mathcal{J}(e_{22},e_{12},[e_{11},e_{21}]_{\alpha}).$$
So, $(B(4,2),[-,-]_{\alpha})$ does not satisfy the Malcev super-identity.
\end{exa}
\section{Hom-Malcev-Admissible Superalgebras}
In Theorem $(\ref{hAandMalcev admissible})$ we showed that every Hom-alternative superalgebra is Hom-Malcev-admissible. The purpose of this section is to introduce and study the class of Hom-flexible, Hom-Malcev-admissible superalgebras. We give several characterizations of Hom-flexible superalgebra that are Hom-Malcev-admissible in terms of the cyclic Hom-associator $($Proposition $(\ref{HMALCEV AND FLEX}))$. Then we prove the analogue of the construction results, Theorem
$(\ref{induced A})$ and Theorem $(\ref{derived A})$. We consider examples of Hom-flexible, Hom-Malcev-admissible superalgebras that are neither Hom-alternative nor Hom-Lie-admissible.\\
To state our characterizations of Hom-flexible superalgebras that are Hom-Malcev-admissible, we need the following definitions:
\begin{df}\label{defin-flexible} A \textsf{ Hom-flexible superalgebra} is a triple $(\mathcal{A},\mu,\alpha)$ consisting of $\mathbb{Z}_{2}-$graded vector space $\mathcal{A}$, an even bilinear map $\mu:\mathcal{A}\times\mathcal{A}\longrightarrow \mathcal{A}$ and an even homomorphism\\
$\alpha:\mathcal{A}\longrightarrow \mathcal{A}$ satisfying the \textsf{ Hom-flexible super-identity}, that is for any $x,y,z \in \mathcal{H}(\mathcal{A})$,
\begin{equation}\label{hom-supe-flexi}
   \widetilde{as_{\mathcal{A}}}(x,y,z)+(-1)^{|x||y|+|x||z|+|y||z|}\widetilde{as_{\mathcal{A}}}(z,y,x)=0.
\end{equation}
\end{df}
It follows from the definition and by Lemma $(\ref{propriet ass H-A})$ that Hom-alternative superalgebra is  Hom-flexible superalgebra. Also, when $\alpha=Id$ in Definition $(\ref{defin-flexible})$, we recover the usual notion of flexible superalgebra.
\begin{df}  Let $(\mathcal{A},\mu,\alpha)$ be any Hom-superalgebra. Define the \textsf{cyclic Hom-associator}
$ S_{\mathcal{A}}:\mathcal{A}\times\mathcal{A}\times\mathcal{A}\longrightarrow\mathcal{A}$ as the even multi-linear map
\begin{eqnarray*}
  S_{\mathcal{A}}(x,y,z) &=&\widetilde{as_{\mathcal{A}}}(x,y,z)+(-1)^{|x|(|y|+|z|)}\widetilde{as_{\mathcal{A}}}(y,z,x)
   +(-1)^{|z|(|x|+|y|)}\widetilde{as_{\mathcal{A}}}(z,x,y),
\end{eqnarray*}
for all $x,y,z \in \mathcal{H}(\mathcal{A})$, where $\widetilde{as_{\mathcal{A}}}$ is the Hom-associator $(\ref{ass})$.
\end{df}
We will use the following preliminary observations about the relationship between the cyclic Hom-associator and the Hom-super-Jacobian $(\ref{Hom-super-Jac})$ of the super-commutator Hom-superalgebra $($Definition $(\ref{supercommutator})$)$:$
\begin{lem} Let $(\mathcal{A},\mu,\alpha)$ be a Hom-flexible superalgebra. Then we have
\begin{equation}\label{superS and SJ}
    \widetilde{J}_{\mathcal{A}^{-}}=2 S_{\mathcal{A}},
\end{equation}
where $\mathcal{A}^{-}=(\mathcal{A},[-,-],\alpha)$ is the super-commutator Hom-superalgebra.
\end{lem}
\begin{proof} Assume that $x,y,z \in \mathcal{H}(\mathcal{A})$, by Lemma $(\ref{superja and ass})$ we have
\begin{eqnarray*}
\widetilde{J}_{\mathcal{A}^{-}}(x,y,z)&=& \widetilde{as_{\mathcal{A}}}(x,y,z)+(-1)^{|x|(|y|+|z|)}\widetilde{as_{\mathcal{A}}}(y,z,x)\\
&+&(-1)^{|z|(|x|+|y|)}\widetilde{as_{\mathcal{A}}}(z,x,y)-(-1)^{|x||y|}\widetilde{as_{\mathcal{A}}}(y,x,z)\\
&-&(-1)^{|y||z|}\widetilde{as_{\mathcal{A}}}(x,z,y)-(-1)^{|x||y|+|z|(|x|+|y|)}\widetilde{as_{\mathcal{A}}}(z,y,x)\\
&=&2 S_{\mathcal{A}} \hskip1cm (\text{by Hom-super flexibility}).
\end{eqnarray*}
\end{proof}
The following result gives characterizations of Hom-Malcev-admissible superalgebras in terms of cyclic Hom-associator, assuming Hom-flexibility:
\begin{prop}\label{HMALCEV AND FLEX} Let $(\mathcal{A},\mu,\alpha)$ be a Hom-flexible superalgebra and $\mathcal{A}^{-}=(\mathcal{A},[-,-],\alpha)$ be the super-commutator Hom-superalgebra. Then the following statements are equivalent$:$
\begin{enumerate}
\item  $\mathcal{A}$ is Hom-Malcev-admissible superalgebra.
\item The equality
\begin{eqnarray*}
&&\widetilde{J}_{\mathcal{A}^{-}}(\alpha(x),\alpha(y),[t,z])
+(-1)^{|x||y|+|t|(|x|+|y|)}\widetilde{J}_{\mathcal{A}^{-}}(\alpha(t),\alpha(y),[x,z])\nonumber\\
&=&(-1)^{|t||z|}[\widetilde{J}_{\mathcal{A}^{-}}(x,y,z),\alpha^{2}(t)]
 +(-1)^{|x|(|y|+|z|+|t|)+|t||y|}[\widetilde{J}_{\mathcal{A}^{-}}(t,y,z),\alpha^{2}(x)]
\end{eqnarray*}
holds for all $x,y,z,t \in \mathcal{H}(\mathcal{A})$.
\item The equality
\begin{eqnarray*}
&& S_{\mathcal{A}}(\alpha(x),\alpha(y),[t,z])+(-1)^{|x||y|+|t|(|x|+|y|)}S_{\mathcal{A}}(\alpha(t),\alpha(y),[x,z])\nonumber\\
&=&(-1)^{|t||z|}[S_{\mathcal{A}}(x,y,z),\alpha^{2}(t)]+(-1)^{|x|(|y|+|z|+|t|)+|t||y|}[S_{\mathcal{A}}(t,y,z),\alpha^{2}(x)]
\end{eqnarray*}
holds for all $x,y,z,t \in \mathcal{H}(\mathcal{A})$.
\end{enumerate}
\end{prop}
The following construction results for Hom-flexible and Hom-Malcev-admissible superalgebras are the analogues of Theorems $(\ref{induced A})$ and
$(\ref{derived A})$.
\begin{thm}\label{inducedMALCEV-ADMISSI} Let $(\mathcal{A},\mu,\alpha)$ be a Hom-Malcev-admissible superalgebra and $\beta: \mathcal{A} \longrightarrow \mathcal{A}$ be an even Malcev-admissible superalgebra endomorphism. Then $\mathcal{A}_{\beta}=(\mathcal{A},\mu_{\beta}=\beta \circ \mu,\beta\alpha)$ is a Hom-Malcev-admissible superalgebra.
\end{thm}
\begin{proof} The super-commutator Hom-superalgebra of $(\mathcal{A},\mu,\alpha)$ is $\mathcal{A}^{-}=(\mathcal{A},[-,-],\alpha)$,
where $[x,y]=\mu(x,y)-(-1)^{|x||y|}\mu(y,x)$ which is a Hom-Malcev superalgebra by assumption. In particular, the Malcev super-identity
\begin{eqnarray}\label{malsev-super-iden}
 &&\mathcal{J}_{\mathcal{A}}(x,y,[t,z])+(-1)^{|x||y|+|t|(|x|+|y|)}\mathcal{J}_{\mathcal{A}}(t,y,[x,z])\nonumber\\
&=&(-1)^{|t||z|}[\mathcal{J}_{\mathcal{A}}(x,y,z),t]
+(-1)^{|x|(|y|+|z|+|t|)+|t||y|}[\mathcal{J}_{\mathcal{A}}(t,y,z),x]
\end{eqnarray}
holds. The super-commutator Hom-superalgebra of the induced Hom-superalgebra $\mathcal{A}_{\beta}$ is $\mathcal{A}_{\beta}^{-}=(\mathcal{A},[-,-]_{\beta},\beta\alpha)$, where
\begin{eqnarray}\label{croch indu}
[x,y]_{\beta} &=& \mu_{\beta}(x,y)-(-1)^{|x||y|}\mu_{\beta}(y,x)\nonumber\\
&=&\beta \circ [x,y]
\end{eqnarray}
which is super skewsymmetric. We must show that $\mathcal{A}_{\beta}^{-}$ satisfies the Hom-Malcev super-identity. We have
\begin{eqnarray*}
  \widetilde{J}_{\mathcal{A}_{\beta}^{-}}(x,y,z)
  &=&[[x,y]_{\beta},\beta\alpha(z)]_{\beta}-[\beta\alpha(x),[y,z]_{\beta}]_{\beta}
  -(-1)^{|y||z|}[[x,z]_{\beta},\beta\alpha(y)]_{\beta}\\
  &=& \beta^{2}\circ \widetilde{J}_{\mathcal{A}^{-}}(x,y,z).
\end{eqnarray*}
Therefore, we have:
\begin{eqnarray*}
&& \widetilde{J}_{\mathcal{A}_{\beta}^{-}}(\beta\alpha(x),\beta\alpha(y),[t,z]_{\beta})
 +(-1)^{|x||y|+|t|(|x|+|y|)}\widetilde{J}_{\mathcal{A}_{\beta}^{-}}(\beta\alpha(t),\beta\alpha(y),[x,z]_{\beta})\\
 &&= \beta^{2}\circ \Big(\widetilde{J}_{\mathcal{A}^{-}}(\beta\alpha(x),\beta\alpha(y),\beta([t,z]))
+(-1)^{|x||y|+|t|(|x|+|y|)}\widetilde{J}_{\mathcal{A}^{-}}(\beta\alpha(t),\beta\alpha(y),\beta([x,z]))\Big)\\
 &&= \beta^{3}\circ \Big(\widetilde{J}_{\mathcal{A}^{-}}(\alpha(x),\alpha(y),[t,z])
 +(-1)^{|x||y|+|t|(|x|+|y|)}\widetilde{J}_{\mathcal{A}^{-}}(\alpha(t),\alpha(y),[x,z])\Big)\\
 &&=\beta^{3}\circ \Big((-1)^{|t||z|}[\widetilde{J}_{\mathcal{A}^{-}}(x,y,z),\alpha^{2}(t)]
 +(-1)^{|x|(|y|+|z|+|t|)+|t||y|}[\widetilde{J}_{\mathcal{A}^{-}}(t,y,z),\alpha^{2}(x)]\Big)\\
 &&= (-1)^{|t||z|}[\beta^{2}\circ\widetilde{J}_{\mathcal{A}^{-}}(x,y,z),(\beta\alpha)^{2}(t)]_{\beta}
 +(-1)^{|x|(|y|+|z|+|t|)+|t||y|}[\beta^{2}\circ\widetilde{J}_{\mathcal{A}^{-}}(t,y,z),(\beta\alpha)^{2}(x)]_{\beta}\\
 &&= (-1)^{|t||z|}[\widetilde{J}_{\mathcal{A}_{\beta}^{-}}(x,y,z),(\beta\alpha)^{2}(t)]_{\beta}
 +(-1)^{|x|(|y|+|z|+|t|)+|t||y|}[\widetilde{J}_{\mathcal{A}_{\beta}^{-}}(t,y,z),(\beta\alpha)^{2}(x)]_{\beta}.
\end{eqnarray*}
This shows that the Hom-Malcev super-identity holds in $\mathcal{A}_{\beta}^{-}$.
\end{proof}
\begin{thm}\label{derivMALCEV-ADM} Let $(\mathcal{A},\mu,\alpha)$ be a Hom-Malcev-admissible superalgebra. Then the derived Hom-superalgebra $\mathcal{A}^{n}=(\mathcal{A},\mu^{(n)}=\alpha^{2^{n}-1} \circ \mu,\alpha^{2^{n}})$ is also a Hom-Malcev-admissible superalgebra for each $n \geq 0$.
\end{thm}
\begin{proof} Assume that $(\mathcal{A},\mu,\alpha)$ is a Hom-Malcev-admissible superalgebra. Note that the super-commutator Hom-superalgebra of the \textsf{nth} derived Hom-superalgebra $\mathcal{A}^{n}$ is $(\mathcal{A}^{n})^{-}=(\mathcal{A},[-,-]^{(n)},\alpha^{2^{n}})$, where
\begin{eqnarray*}
  [x,y]^{(n)} &=&\mu^{(n)}(x,y)-(-1)^{|x||y|}\mu^{(n)}(y,x) \\
   &=&\alpha^{2^{n}-1} \circ (\mu(x,y)-(-1)^{|x||y|}\mu(y,x))\\
   &=&\alpha^{2^{n}-1} \circ [x,y].
\end{eqnarray*}
Thus, we have
\begin{equation}\label{HMALCEV ADM}
    (\mathcal{A}^{n})^{-}=(\mathcal{A}^{-})^{n}
\end{equation}
where $\mathcal{A}^{-}=(\mathcal{A},[-,-],\alpha)$ is the super-commutator Hom-superalgebra of $\mathcal{A}$ and $(\mathcal{A}^{-})^{n}$ is its \textsf{nth} derived Hom-superalgebra. Since $[-,-]^{(n)}$ is super skewsymmetric. We must show that $(\mathcal{A}^{n})^{-}$ satisfies the Hom-Malcev super-identity. To do that, observe that
\begin{eqnarray}\label{hma-admissi-der}
\widetilde{J}_{(\mathcal{A}^{n})^{-}}&=&\widetilde{J}_{(\mathcal{A}^{-})^{n}}\nonumber\\
&=&\alpha^{2^{n}-1} \circ \widetilde{J}_{\mathcal{A}^{-}}.
\end{eqnarray}
Using $(\ref{Jacobian rong alpha})$ and $(\ref{hma-admissi-der})$ we compute as follows
\begin{eqnarray*}
 &&\widetilde{J}_{(\mathcal{A}^{n})^{-}}(\alpha^{2^{n}}(x),\alpha^{2^{n}}(y),[t,z]^{(n)})
 +(-1)^{|x||y|+|t|(|x|+|y|)}\widetilde{J}_{(\mathcal{A}^{n})^{-}}(\alpha^{2^{n}}(t),\alpha^{2^{n}}(y),[x,z]^{(n)})\\
 &&=\alpha^{2^{n}-1}\circ \Big(\widetilde{J}_{\mathcal{A}^{-}}(\alpha^{2^{n}}(x),\alpha^{2^{n}}(y),\alpha^{2^{n}-1}([t,z]))\\
&&\ +(-1)^{|x||y|+|t|(|x|+|y|)}\widetilde{J}_{\mathcal{A}^{-}}(\alpha^{2^{n}}(t),\alpha^{2^{n}}(y),\alpha^{2^{n}-1}([x,z]))\Big)\\
&&=\alpha^{3(2^{n}-1)}\circ \Big(\widetilde{J}_{\mathcal{A}^{-}}(\alpha(x),\alpha(y),[t,z])
  + (-1)^{|x||y|+|t|(|x|+|y|)}\widetilde{J}_{\mathcal{A}^{-}}(\alpha(t),\alpha(y),[x,z])\Big)\\
 &&= (-1)^{|t||z|}[\widetilde{J}_{(\mathcal{A}^{n})^{-}}(x,y,z),(\alpha^{2^{n}})^{2}(t)]^{(n)}
 + (-1)^{|x|(|y|+|z|+|t|)+|t||y|}[\widetilde{J}_{(\mathcal{A}^{n})^{-}}(t,y,z),(\alpha^{2^{n}})^{2}(x)]^{(n)}.
\end{eqnarray*}
We have shown that $(\mathcal{A}^{n})^{-}$ satisfies the Hom-Malcev super-identity. So $\mathcal{A}^{n}$ is Hom-Malcev-admissible superalgebra.
\end{proof}
\begin{thm}\label{induced and deriv flexi}\  \begin{enumerate}
\item  Let $(\mathcal{A},\mu,\alpha)$ be a Hom-flexible superalgebra and $\beta: \mathcal{A} \longrightarrow \mathcal{A}$ be an even flexible superalgebra
endomorphism. Then $\mathcal{A}_{\beta}=(\mathcal{A},\mu_{\beta}=\beta \circ \mu,\beta\alpha)$ is a Hom-flexible superalgebra.
\item Let $(\mathcal{A},\mu,\alpha)$ be a Hom-flexible superalgebra. Then the derived Hom-superalgebra \\$\mathcal{A}^{n}=(\mathcal{A},\mu^{(n)}=\alpha^{2^{n}-1} \circ \mu,\alpha^{2^{n}})$ is also a Hom-flexible superalgebra for each $n \geq 0$.
\end{enumerate}
\end{thm}
\begin{proof} For the first assertion, for any superalgebra $(\mathcal{A},\mu)$, we regard it as the Hom-superalgebra $(\mathcal{A},\mu, Id)$ with identity
twisting map. Then we have:
\begin{eqnarray*}
 \widetilde{ as_{\mathcal{A}_{\beta}}}(x,y,z) &=&\mu_{\beta}(\mu_{\beta}(x,y),\beta\alpha(z))-\mu_{\beta}(\beta\alpha(x),\mu_{\beta}(y,z)\\
   &=& \beta^{2}\circ (\mu(\mu(x,y),\alpha(z))-\mu(\alpha(x),\mu(y,z))) \\
   &=& \beta^{2}\circ \widetilde{as_{\mathcal{A}}}(x,y,z). \end{eqnarray*}
Then  $$\widetilde{as_{\mathcal{A}_{\beta}}}=\beta^{2}\circ \widetilde{as_{\mathcal{A}}}.$$
Now for a flexible superalgebra $(\mathcal{A},\mu)$, this implies that
\begin{eqnarray*}
\widetilde{as_{\mathcal{A}_{\beta}}}(x,y,z)&=&\beta^{2} \circ \widetilde{as_{\mathcal{A}}}(x,y,z)\\
                                &=&-(-1)^{|x||y|+|x||z|+|y||z|}\beta^{2} \circ \widetilde{as_{\mathcal{A}}}(z,y,x)\\
                                &=&-(-1)^{|x||y|+|x||z|+|y||z|} as_{\mathcal{A}_{\beta}}(z,y,x),\end{eqnarray*}
then $$\widetilde{as_{\mathcal{A}_{\beta}}}(x,y,z) +(-1)^{|x||y|+|x||z|+|y||z|} \widetilde{as_{\mathcal{A}_{\beta}}}(z,y,x)=0.$$
So $\mathcal{A}_{\beta}$ is Hom-flexible superalgebra.
\end{proof}
\subsection{Examples of Hom-flexible superalgebras} We construct examples of Hom-flexible superalgebras using Theorem $(\ref{induced and deriv flexi})$.
\begin{exa}\label{exp1-flex}$($Hom-flexible superalgebra of dimension $3)$. We consider the simple flexible superalgebra
$K_{3}(\beta,\gamma,\eta)=(K_{3})_{0}\oplus (K_{3})_{1}$ (see \cite{shest-Non-commutative}), where $(K_{3})_{0}=span\{e_{1}\}$ and $(K_{3})_{1}=span\{e_{2},e_{3}\},$
endowed with a product given by the following multiplication table$:$
$$\begin{tabular}{|l|l|l|l|}
  \hline
   $\mu$ & $e_{1}$ & $e_{2}$ & $e_{3}$ \\ \hline
   $e_{1}$& $e_{1}$ &$\beta e_{2}+\gamma e_{3}$ & $(1-\beta) e_{3}+\eta e_{2}$ \\\hline
   $e_{2}$&$(1-\beta) e_{2}-\gamma e_{3}$  & $-2\gamma e_{1}$ &$2\beta e_{1}$  \\\hline
  $e_{3}$ &$-\eta e_{2}+\beta e_{3}$  & $-2(1-\beta) e_{1}$ & $2\eta e_{1}$ \\
  \hline
\end{tabular}$$
Even superalgebra endomorphisms $\alpha$ with respect to the basis $\{e_{1},e_{2},e_{3}\}$ are defined by the matrix
$$\alpha=\left(
  \begin{array}{ccc}
    1 & 0 & 0 \\
    0 & a & \frac{\eta b}{\gamma} \\
    0 & \frac{-a+\sqrt{-4 \gamma \eta+a^{2}+4\gamma \eta a^{2}}}{2 \eta} & \frac{\gamma a-b}{\gamma} \\
  \end{array}
\right)$$
with $\beta=1$ and $\gamma,\eta\neq 0$. According to  Theorem $(\ref{induced and deriv flexi})$,
the even linear map $\alpha$ an the following multiplication
\begin{eqnarray*}
&& \mu_{\alpha}(e_{1},e_{1})=e_{1},~~
\mu_{\alpha}(e_{1},e_{2})=\Big(a(1-\beta)-\eta b\Big)e_{2}+\Big((1-\beta)\frac{-a+\sqrt{-4 \gamma \eta+a^{2}+4\gamma \eta a^{2}}}{2 \eta}+ (b-\gamma a)\Big) e_{3},
\\ &&  \mu_{\alpha}(e_{1},e_{3})=\Big(\frac{b \beta \eta}{\gamma}- \eta a\Big)e_{2}+\Big(\frac{a-\sqrt{-4 \gamma \eta+a^{2}+4\gamma \eta a^{2}}}{2}
+\frac{\beta(\gamma a-b)}{\gamma}\Big) e_{3},
\\ &&  \mu_{\alpha}(e_{2},e_{1})=\Big(\beta a+\gamma b\Big)e_{2}+\Big((\gamma a-b)+\beta \frac{-a+\sqrt{-4 \gamma \eta+a^{2}+4\gamma \eta a^{2}}}{2 \eta}\Big)e_{3},
\\ &&  \mu_{\alpha}(e_{2},e_{2})=-2\gamma e_{1},~~\mu_{\alpha}(e_{2},e_{3})=-2(1- \beta) e_{1},
\\ &&  \mu_{\alpha}(e_{3},e_{1})=(\eta a+\frac{(1-\beta) b\eta}{\gamma})e_{2}+ \Big(\frac{(1-\beta)(\gamma a-b)}{\gamma}
+(\frac{-a+\sqrt{-4 \gamma \eta+a^{2}+4\gamma \eta a^{2}}}{2 })\Big)e_{3},
\\ &&  \mu_{\alpha}(e_{3},e_{2})=2\beta e_{1},~~
\mu_{\alpha}(e_{3},e_{3})=2 \eta e_{1},
\end{eqnarray*}
where $a,b \in \mathbb{K}$ and $\gamma,\eta \neq 0$, determine  $3$-dimensional Hom-flexible, Hom-Malcev-admissible superalgebra.
\end{exa}

\begin{exa}\label{exp2-flexi} $($Hom-flexible superalgebras of dimension $4)$. Let $\beta \in \mathbb{K}$ and $t\neq 0$. Define a superalgebra $U=D_{t}(\beta)$ (see \cite{shest-Non-commutative}), where $U_{0}=span\{e_{1},e_{2}\}$ and  $U_{1}=span\{x,y\}$ by setting
$$\begin{tabular}{|l|l|l|l|l|}
  \hline
  $\mu$ & $e_{1}$ & $e_{2}$ & $x$ & $y$ \\ \hline
  $e_{1}$ & $e_{1}$ & $0$ & $(1-\beta)y$ & $\beta y$ \\ \hline
  $e_{2}$ & $0$ & $e_{2}$ & $\beta x$ & $(1-\beta)y$ \\ \hline
  $x$ & $\beta x$ & $\beta x$ & $0$ & $-2((1-\beta) e_{1}+\beta t e_{2})$ \\ \hline
  $y$ & $(1-\beta)y$ & $\beta y$ & $2(\beta e_{1}+(1-\beta)t e_{2})$ & $0$ \\
  \hline
\end{tabular}$$
Then $(U,\mu)$ is a simple flexible, Malcev-admissible superalgebra.\\
Even superalgebra endomorphisms $\alpha$ with respect to the basis $\{e_{1},e_{2},x,y\}$ are defined by the matrix
$$\alpha=\left(
          \begin{array}{cccc}
            1 & 0 & 0 & 0 \\
            0 & 2 & 0 & 0 \\
            0 & 0 & a & 0 \\
            0 & 0 & 0 & \frac{1}{a} \\
          \end{array}
        \right)$$
where $a\neq 0$. According to  Theorem $(\ref{induced and deriv flexi})$, the even linear map $\alpha$ an the following multiplication table$:$
$$\begin{tabular}{|l|l|l|l|l|}
  \hline
  $\mu_{\alpha}$ & $e_{1}$ & $e_{2}$ & $x$ & $y$ \\ \hline
  $e_{1}$ & $e_{1}$ & $0$ & $\frac{(1-\beta)}{a}y$ & $\frac{\beta}{a} y$ \\ \hline
  $e_{2}$ & $0$ & $e_{2}$ & $\beta a x$ & $\frac{(1-\beta)}{a}y$ \\ \hline
  $x$ & $\beta a x$ & $\beta a x$ & $0$ & $-2((1-\beta) e_{1}+\beta t e_{2})$ \\ \hline
  $y$ & $\frac{(1-\beta)}{a}y$ & $\frac{\beta}{a} y$ & $2(\beta e_{1}+(1-\beta)t e_{2})$ & $0$ \\
  \hline
\end{tabular}$$
where $a\neq 0$, determine $4$-dimensional Hom-flexible, Hom-Malcev-admissible superalgebra.\\
Note that $U_{\alpha}$ is not Hom-alternative superalgebra because
\begin{eqnarray*}
 \widetilde{as_{U_{\alpha}}}(e_{2},x,y)+\widetilde{as_{U_{\alpha}}}(x,e_{2},y)&=& 2\beta^{2}e_{1}+(\beta-2)(1-\beta)t e_{2}
\end{eqnarray*}
$\widetilde{as_{U_{\alpha}}}(e_{2},x,y)+\widetilde{as_{U_{\alpha}}}(x,e_{2},y)\neq 0$ whenever $\beta \neq 0,1,2$. So $(U,\mu_{\alpha},\alpha)$ does not satisfy the Hom-alternative super-identity, therefore, is not Hom-alternative.\\
Also, $U_{\alpha}$ is not Hom-Lie-admissible superalgebra because
\begin{eqnarray*}
 \hskip1cm \widetilde{ J}_{(U_{\alpha})^{-}}(e_{1},e_{2},x) = \frac{(1-\beta)(2\beta-1)}{a^{2}}y
\end{eqnarray*}
$\widetilde{ J}_{(U_{\alpha})^{-}}(e_{1},e_{2},x)\neq 0$ whenever $\beta\neq 1,\frac{1}{2}$ and $a\neq 0$.\\
Finally $(U,\mu_{\alpha})$ is not Malcev-admissible, i.e., $(U,[-,-]_{\alpha})$ is not Malcev superalgebra. Indeed, let $\mathcal{J}$ the usual super-Jacobian of $(U,[-,-]_{\alpha})$ as in $(\ref{sup-Jacobian})$. Then, on the one hand, we have
\begin{eqnarray*}
  \mathcal{J}(e_{1},e_{2},[x,y]_{\alpha})+\mathcal{J}(x,e_{2},[e_{1},y]_{\alpha}) &=&[[e_{1},e_{2}]_{\alpha},[x,y]_{\alpha}]_{\alpha}-[e_{1},[e_{2},[x,y]_{\alpha}]_{\alpha}- [e_{1},[x,y]_{\alpha}]_{\alpha},e_{2}]_{\alpha} \\
   &+& [[x,e_{2}]_{\alpha},[e_{1},y]_{\alpha}]_{\alpha}-[x,[e_{2},[e_{1},y]_{\alpha}]_{\alpha}- [x,[e_{1},y]_{\alpha}]_{\alpha},e_{2}]_{\alpha} \\
   &=&\frac{2(1-\beta)(1-2\beta)}{a^{2}}e_{1}+\frac{2t(2-\beta)(2\beta-1)}{a^{2}}e_{2}.
\end{eqnarray*}
On the other hand, we have
\begin{eqnarray*}
  [\mathcal{J}(e_{1},e_{2},y),x]_{\alpha}+[\mathcal{J}(x,e_{2},y),e_{1}]_{\alpha} &=&[[[e_{1},e_{2}]_{\alpha},y]_{_{\alpha}},x]_{\alpha}-[[e_{1},[e_{2},y]_{\alpha}]_{\alpha},x]_{\alpha}
  -[[[e_{1},y]_{\alpha},e_{2}]_{_{\alpha}},x]_{\alpha}\\
  &+& [[[x,e_{2}]_{\alpha},y]_{_{\alpha}},e_{1}]_{\alpha}-[[x,[e_{2},y]_{\alpha}]_{\alpha},e_{1}]_{\alpha}
  -[[[x,y]_{\alpha},e_{2}]_{_{\alpha}},e_{1}]_{\alpha}\\
   &=& \frac{4\beta(\beta-1)(2\beta-1)}{a^{2}}e_{1}+\frac{4\beta t(\beta-1)(1-2\beta)}{a^{2}}e_{2}.
\end{eqnarray*}
 $\mathcal{J}(e_{1},e_{2},[x,y]_{\alpha})+\mathcal{J}(x,e_{2},[e_{1},y]_{\alpha})\neq [\mathcal{J}(e_{1},e_{2},y),x]_{\alpha}
 +[\mathcal{J}(x,e_{2},y),e_{1}]_{\alpha}$ whenever $\beta \neq 0,\frac{1}{2},1,2$ and $a\neq 0$. So $(U,[-,-]_{\alpha})$ does not satisfy the Malcev super-identity and, therefore, is not a Malcev superalgebra. Then $(U,\mu_{\alpha})$ does not in general a Malcev-admissible superalgebra.
\end{exa}
\begin{center}\section{Hom-Alternative Superalgebras are Hom-Jordan-admissible}\end{center}
 We introduce in this section  Hom-Jordan $(-$admissible$)$ superalgebras and we show in Theorem $(\ref{has})$ that every Hom-alternative superalgebras are Hom-Jordan-admissible. Then we prove Theorems $(\ref{iduced})$ and $(\ref{admissible})$, which are constructions results for Hom-Jordan and Hom-Jordan admissible superalgebras. In Examples $(\ref{ex-jordan1})$ and $(\ref{ex-jorda2})$ we construct Hom-Jordan superalgebras from the $3$-dimensional Kaplansky Jordan superalgebra and from the $4$-dimensional simple Jordan superalgebra $D_{t}$ where $t\neq 0$.\\
Let us begin with some relevant definitions.
\begin{df}\label{plus hom} Let $(\mathcal{A},\mu,\alpha)$ be a Hom-superalgebra. Define its \textsf{plus} Hom-superalgebras as the Hom-superalgebra
$\mathcal{A}^{+}=(\mathcal{A},\ast,\alpha)$, where the product $\ast$ is given by
$$x\ast y= \frac{1}{2}(\mu(x,y)+(-1)^{|x||y|}\mu(y,x))=\frac{1}{2}(xy+(-1)^{|x||y|}yx),$$
with $\mu(x,y)=xy$. The product $\ast$ is super-commutative.
\end{df}
\begin{df}\   \begin{enumerate}
\item  A \textsf{Hom-Jordan superalgebra} is a Hom-superalgebra $(\mathcal{A},\mu,\alpha)$ such that $\mu$ is super-commutative
$(i.e. \mu(x,y)=(-1)^{|x||y|}\mu(y,x))$and the \textsf{Hom-Jordan super-identity}
\begin{eqnarray}\label{HOM-SUPER-IDE}
&&\sum \limits _{x,y,t}(-1)^{|t|(|x|+|z|)} \widetilde{as_{\mathcal{A}}}(x y,\alpha(z),\alpha(t))=0
\end{eqnarray}
is satisfied for all $x,y,z$ and $t$ in $\mathcal{H}(\mathcal{A})$, where  $\sum \limits _{x,y,t}$ denotes the cyclic sum over $(x,y,t)$ and $\widetilde{as_{\mathcal{A}}}$ is the Hom-associator $(\ref{ass})$.
 \item  A \textsf{Hom-Jordan-admissible superalgebra} is a Hom-superalgebra $(\mathcal{A},\mu,\alpha)$ whose plus Hom-superalgebra $\mathcal{A}^{+}=(\mathcal{A},\ast,\alpha)$ is a Hom-Jordan superalgebra.
 \end{enumerate}
 The Hom-Jordan super-identity $(\ref{HOM-SUPER-IDE})$ can be rewritten as
 \begin{align*}
\sum \limits _{x,y,t}(-1)^{|t|(|x|+|z|)}\Big(\mu(\mu(\mu(x,y),\alpha(z)),\alpha^{2}(t))-\mu(\alpha(\mu(x,y)),\mu(\alpha(z),\alpha(t)))\Big)
=0.
\end{align*}
 Since the product $\ast$ is super-commutative, a Hom-superalgebra $(\mathcal{A},\mu,\alpha)$ is Hom-Jordan-admissible if
 and only if $\mathcal{A}^{+}=(\mathcal{A},\ast,\alpha)$ satisfies the Hom-Jordan super-identity
\begin{eqnarray}\label{hjsuper}
\sum \limits _{x,y,t}(-1)^{|t|(|x|+|z|)} \widetilde{as_{\mathcal{A}^{+}}}(x \ast y,\alpha(z),\alpha(t)) =0
\end{eqnarray}
for all $x,y,z$ and $t$ in $\mathcal{H}(\mathcal{A})$, where $\sum \limits _{x,y,t}$ denotes summations over the cyclic permutation on $x,y,t$.\\
Or equivalently
\begin{align*}
&\sum \limits _{x,y,t}(-1)^{|t|(|x|+|z|)}\Big(((x \ast y)\ast \alpha(z))\ast \alpha^{2}(t)-\alpha(x \ast y)\ast(\alpha(z)\ast \alpha(t))\Big)
=0
\end{align*}
for all $x,y,z$ and $t$ in $\mathcal{H}(\mathcal{A}),$ where $\sum \limits _{x,y,t}$ denotes the cyclic sum over $(x,y,t)$.
\end{df}
\begin{exa} A Jordan $($-admissible$)$ superalgebra is a Hom-Jordan $(-$admissible$)$ superalgebra with $\alpha=Id$, since the Hom-Jordan super-identity
$(\ref{hjsuper})$ with $\alpha=Id$ is the Jordan super-identity $(\ref{SUPER-IDE-JORDAN})$. We  refer to \cite{helena-super-jor-alte-malcev-with,Kac-Jordan-super1977,Martinez-Jordan,Racine-classi-jordan-super,shest-jordan-super,shtern}
for discussions about structure of Jordan superalgebras.
\end{exa}
\begin{rem} In \cite{makhlouf hom-altern2009hom}, a Hom-Jordan algebra is defined as a commutative Hom-algebra satisfying\\
 $\sum \limits _{x,y,t}\widetilde{as_{\mathcal{A}}}(xy,z,\alpha(t))=0$, which becomes our Hom-Jordan super-identity $(\ref{hjsuper})$ if $z$ is replaced by $\alpha(z)$. Using this  definition of a Hom-Jordan algebra, Hom-alternative superalgebras are not Hom-Jordan-admissible.
\end{rem}
\begin{thm}\label{has} Let $(\mathcal{A},\mu,\alpha)$ be a multiplicative Hom-alternative superalgebra. Then $\mathcal{A}$ is Hom-Jordan-admissible, in the sense that
the Hom-superalgebra
$$\mathcal{A}^{+}=(\mathcal{A},\ast,\alpha)$$
is a multiplicative Hom-Jordan superalgebra, where $x \ast y=\frac{1}{2}(x y+(-1)^{|x||y|}y x).$
\end{thm}
To prove Theorem $(\ref{has})$, we will use the following preliminary observation.
\begin{lem}\label{uti} Let $(\mathcal{A},\mu,\alpha)$ be any Hom-superalgebra and $\mathcal{A}^{+}=(\mathcal{A},\ast,\alpha)$ be its plus Hom-superalgebra. Then we have
\begin{eqnarray}\label{id1}
&&\sum \limits _{x,y,t}4(-1)^{|t|(|x|+|z|)} \widetilde{as_{\mathcal{A}^{+}}}(x \ast y,\alpha(z),\alpha(t)) \nonumber\\
&& \ \ =\sum \limits _{x,y,t}\Big((-1)^{|t|(|x|+|z|)} \widetilde{as_{\mathcal{A}}}(x y,\alpha(z),\alpha(t))+(-1)^{|t|(|x|+|z|)+|x||y|} \widetilde{as_{\mathcal{A}}}(y x,\alpha(z),\alpha(t))\nonumber\\
&&\ \ \ -(-1)^{|z|(|x|+|y|)+|y|(|x|+|t|)} \widetilde{as_{\mathcal{A}}}(\alpha(t),\alpha(z),y x)-(-1)^{|z|(|x|+|y|)+|y|(|x|+|t|)} \widetilde{as_{\mathcal{A}}}(\alpha(t),\alpha(z),y x)\nonumber\\
&&\ \ \ -(-1)^{|t||y|} \widetilde{as_{\mathcal{A}}}(\alpha(t),x y,\alpha(z))-(-1)^{|y|(|x|+|t|)} \widetilde{as_{\mathcal{A}}}( \alpha(t),y x,\alpha(z))\nonumber\\
&&\ \ \  +(-1)^{|x|(|y|+|t|)+|z|(|x|+|y|+|t|)} \widetilde{as_{\mathcal{A}}}(\alpha(z),y x,\alpha(t))+(-1)^{|t||x|+|z|(|x|+|y|+|t|)} \widetilde{as_{\mathcal{A}}}(\alpha(z),x y,\alpha(t))\nonumber\\
&&\ \ \  -(-1)^{|z|(|x|+|y|+|t|)+|t||y|} \widetilde{as_{\mathcal{A}}}(\alpha(z),\alpha(t),x y)+(-1)^{|t||x|} \widetilde{as_{\mathcal{A}}}(x y,\alpha(t),\alpha(z))\nonumber\\
&&\ \ \ -(-1)^{|z|(|x|+|y|+|t|)+|y|(|x|+|z|)} \widetilde{as_{\mathcal{A}}}(\alpha(z),\alpha(t),x y) +(-1)^{|x|(|y|+|t|)} \widetilde{as_{\mathcal{A}}}(y x,\alpha(t),\alpha(z))\Big)\\
&&\ \ \ +(-1)^{|x||y|+|z|(|x|+|y|+|t|)}[\alpha^{2}(z),\widetilde{as_{\mathcal{A}}}(y,t,x)]+(-1)^{|y|(|x|+|t|)+|z|(|x|+|y|+|t|)}[\alpha^{2}(z),\widetilde{as_{\mathcal{A}}}(t,y,x)]\nonumber\\
&&\ \ \ +(-1)^{|t||y|+|z|(|x|+|y|+|t|)}[\alpha^{2}(z),\widetilde{as_{\mathcal{A}}}(t,x,y)]+(-1)^{|t|(|x|+|y|)+|z|(|x|+|y|+|t|)}[\alpha^{2}(z),\widetilde{as_{\mathcal{A}}}(x,t,y)]\nonumber\\
&&\ \ \ +(-1)^{|x||t|+|z|(|x|+|y|+|t|)}[\alpha^{2}(z),\widetilde{as_{\mathcal{A}}}(x,y,t)]+(-1)^{|x|(|y|+|t|)+|z|(|x|+|y|+|t|)}[\alpha^{2}(z),\widetilde{as_{\mathcal{A}}}(y,x,t)]\nonumber
\end{eqnarray}
for all $x,y,z,t \in \mathcal{H}(\mathcal{A})$, where $[-,-]$ is the super-commutator bracket\\ $($i.e. $[x,y]=xy-(-1)^{|x||y|}yx$ $)$ and where
$\sum \limits _{x,y,t}$ denotes the cyclic sum over $(x,y,t)$.
\end{lem}
\begin{proof} As usual we write $\mu(x,y)$ as the juxtaposition $xy$. Starting from the left-hand side of $(\ref{id1})$, we have$:$
\begin{eqnarray*}
&&\sum \limits _{x,y,t}4(-1)^{|t|(|x|+|z|)} \widetilde{as_{\mathcal{A}^{+}}}(x \ast y,\alpha(z),\alpha(t))\nonumber\\
&&\ \ =\sum \limits _{x,y,t}4 (-1)^{|t|(|x|+|z|)}\Big(((x \ast y)\ast \alpha(z))\ast \alpha^{2}(t)-\alpha(x \ast y)\ast(\alpha(z)\ast \alpha(t))\Big) \nonumber\\
&&\ \ = \sum \limits _{x,y,t}\Big( (-1)^{|t|(|x|+|z|)}((x y)\alpha(z))\alpha^{2}(t)+(-1)^{|t||y|}\alpha^{2}(t)((x y)\alpha(z))\nonumber\\
&&\ \ \ +(-1)^{|t|(|x|+|z|)+|z|(|x|+|y|)}(\alpha(z)(x y))\alpha^{2}(t)+(-1)^{|t||y|+|z|(|x|+|y|)}\alpha^{2}(t)(\alpha(z)(x y))\nonumber\\
&&\ \ \ +(-1)^{|t|(|x|+|z|)+|x||y|}((y x)\alpha(z))\alpha^{2}(t)+(-1)^{|y|(|x|+|t|)}\alpha^{2}(t)((y x)\alpha(z))\nonumber\\
&&\ \ \ +(-1)^{|t|(|x|+|z|)+|y|(|x|+|t|)}\alpha^{2}(t)(\alpha(z)(y x))+(-1)^{|x|(|t|+|y|)+|z|(|x|+|y|+|t|)}(\alpha(z)(y x))\alpha^{2}(t)\nonumber\\
&&\ \ \ -(-1)^{|t|(|x|+|z|)}\alpha(x y)(\alpha(z)\alpha(t))-(-1)^{|t||y|+|z|(|x|+|y|+|t|)}(\alpha(z)\alpha(t))\alpha(x y)\nonumber\\
&&\ \ \ -(-1)^{|t||x|}\alpha(x y)(\alpha(t)\alpha(z))-(-1)^{|t||y|+|z|(|x|+|y|)}(\alpha(t)\alpha(z))\alpha(x y)\nonumber\\
&&\ \ \ -(-1)^{|t|(|x|+|z|)+|x||y|}\alpha(y x)(\alpha(z)\alpha(t))-(-1)^{|y|(|t|+|x|)+|z|(|x|+|y|+|t|)}(\alpha(z)\alpha(t))\alpha(y x)\nonumber\\
&&\ \ \ -(-1)^{|x|(|y|+|t|)}\alpha(y x)(\alpha(t)\alpha(z))-(-1)^{|y|(|t|+|x|)+|z|(|x|+|y|)}(\alpha(t)\alpha(z))\alpha(y x)\nonumber\\
&&\ \ =(-1)^{|t|(|x|+|z|)} \widetilde{as_{\mathcal{A}}}(x y,\alpha(z),\alpha(t))-(-1)^{|t||y|+|z|(|x|+|y|)}\widetilde{as_{\mathcal{A}}}(\alpha(t),\alpha(z),x y)\nonumber\\
\end{eqnarray*}
\begin{eqnarray}\label{id2}
&&\ \ \ +(-1)^{|t|(|x|+|z|)+|x||y|}\widetilde{as_{\mathcal{A}}}(y x,\alpha(z),\alpha(t))-(-1)^{|y|(|x|+|t|)+|z|(|x|+|y|)}\widetilde{as_{\mathcal{A}}}(\alpha(t),\alpha(z),y x)\nonumber\\
&&\ \ \ +(-1)^{|t||y|}\alpha^{2}(t)((x y)\alpha(z))+(-1)^{|t|(|x|+|z|)+|z|(|x|+|y|)}(\alpha(z)(x y))\alpha^{2}(t)\nonumber\\
&&\ \ \ +(-1)^{|y|(|x|+|t|)}\alpha^{2}(t)((y x)\alpha(z))+(-1)^{|x|(|t|+|y|)+|z|(|x|+|y|+|t|)}(\alpha(z)(y x))\alpha^{2}(t) \nonumber\\
&&\ \ \ -(-1)^{|t||y|+|z|(|x|+|y|+|t|)}(\alpha(z)\alpha(t))\alpha(x y)-(-1)^{|t||x|}\alpha(x y)(\alpha(t)\alpha(z))\nonumber\\
&&\ \ \ -(-1)^{|y|(|t|+|x|)+|z|(|x|+|y|+|t|)}(\alpha(z)\alpha(t))\alpha(y x)-(-1)^{|x|(|y|+|t|)}\alpha(y x)(\alpha(t)\alpha(z))\Big).
\end{eqnarray}
Using the definition of the Hom-associator $(\ref{ass})$, the last eight terms in $(\ref{id2})$ are$:$
\begin{eqnarray*}
\bullet && \sum \limits _{x,y,t}(-1)^{|x|(|t|+|y|)+|z|(|x|+|y|+|t|)}(\alpha(z)(y x))\alpha^{2}(t)\nonumber\\
&&=\sum \limits _{x,y,t}(-1)^{|x|(|t|+|y|)+|z|(|x|+|y|+|t|)}\Big(\widetilde{as_{\mathcal{A}}}(\alpha(z),y x,\alpha(t))+\alpha^{2}(z)((y x)\alpha(t))\Big)\nonumber\\
\bullet && \sum \limits _{x,y,t}(-1)^{|t||y|}\alpha^{2}(t)((x y)\alpha(z))\nonumber\\
&&=-\sum \limits _{x,y,t}(-1)^{|t||y|}\Big(\widetilde{as_{\mathcal{A}}}(\alpha(t),x y,\alpha(z))-(\alpha(t)(x y))\alpha^{2}(z)\Big)\nonumber\\
\bullet &&  \sum \limits _{x,y,t}(-1)^{|x|(|t|+|y|)+|z|(|x|+|y|+|t|)}(\alpha(z)(y x))\alpha^{2}(t)\nonumber\\
&&=\sum \limits _{x,y,t}(-1)^{|x|(|t|+|y|)+|z|(|x|+|y|+|t|)}\Big(\widetilde{as_{\mathcal{A}}}(\alpha(z),y x,\alpha(t))+\alpha^{2}(z)((y x)\alpha(t))\Big)\nonumber\\
\bullet && \sum \limits _{x,y,t}(-1)^{|t||y|}\alpha^{2}(t)((x y)\alpha(z))\nonumber\\
&&=-\sum \limits _{x,y,t}(-1)^{|t||y|}\Big(\widetilde{as_{\mathcal{A}}}(\alpha(t),x y,\alpha(z))-(\alpha(t)(x y))\alpha^{2}(z)\Big)\nonumber\\
\bullet && \sum \limits _{x,y,t}(-1)^{|t|(|x|+|z|)+|z|(|x|+|y|)}(\alpha(z)(x y))\alpha^{2}(t)\nonumber\\
&&=\sum \limits _{x,y,t}(-1)^{|t|(|x|+|z|)+|z|(|x|+|y|)}\Big(\widetilde{as_{\mathcal{A}}}(\alpha(z),x y,\alpha(t))+\alpha^{2}(z)((x y)\alpha(t))\Big)\nonumber\\
\bullet && \sum \limits _{x,y,t}(-1)^{|y|(|x|+|t|)}\alpha^{2}(t)((y x)\alpha(z))\nonumber\\
&&=-\sum \limits _{x,y,t}(-1)^{|y|(|x|+|t|)}\Big(\widetilde{as_{\mathcal{A}}}(\alpha(t),y x,\alpha(z))-(\alpha(t)(y x))\alpha^{2}(z)\Big)\nonumber\\
\bullet && -\sum \limits _{x,y,t}(-1)^{|t||x|}\alpha(x y)(\alpha(t)\alpha(z))\nonumber\\
&&=\sum \limits _{x,y,t}(-1)^{|t||x|}\Big(\widetilde{as_{\mathcal{A}}}(x y,\alpha(t),\alpha(z))-((x y)\alpha(t))\alpha^{2}(z)\Big)\nonumber\\
\bullet && -\sum \limits _{x,y,t}(-1)^{|y|(|t|+|x|)+|z|(|x|+|y|+|t|)}(\alpha(z)\alpha(t))\alpha(y x)\nonumber\\
&&=-\sum \limits _{x,y,t}(-1)^{|y|(|t|+|x|)+|z|(|x|+|y|+|t|)}\Big(\widetilde{as_{\mathcal{A}}}(\alpha(z),\alpha(t),y x)+\alpha^{2}(z)(\alpha(t)(y x))\Big)\nonumber\\
\bullet && -\sum \limits _{x,y,t}(-1)^{|t||y|+|z|(|x|+|y|+|t|)}(\alpha(z)\alpha(t))\alpha(x y)\nonumber\\
&&=-\sum \limits _{x,y,t}(-1)^{|t||y|+|z|(|x|+|y|+|t|)}\Big(\widetilde{as_{\mathcal{A}}}(\alpha(z),\alpha(t),x y)+\alpha^{2}(z)(\alpha(t)(x y))\Big)\nonumber\\
\end{eqnarray*}
\begin{eqnarray}\label{idasso}
\bullet && -\sum \limits _{x,y,t}(-1)^{|x|(|y|+|t|)}\alpha(y x)(\alpha(t)\alpha(z))\nonumber\\
&&=\sum \limits _{x,y,t}(-1)^{|x|(|y|+|t|)}\Big(\widetilde{as_{\mathcal{A}}}(y x,\alpha(t),\alpha(z))-((y x)\alpha(t))\alpha^{2}(z)\Big),
\end{eqnarray}
or
\begin{eqnarray*}
&&\sum \limits _{x,y,t}\Big((-1)^{|x|(|t|+|y|)+|z|(|x|+|y|+|t|)}\alpha^{2}(z)((y x)\alpha(t))+(-1)^{|t||y|}(\alpha(t)(x y))\alpha^{2}(z)\\
&&+(-1)^{|t|(|x|+|z|)+|z|(|x|+|y|)}\alpha^{2}(z)((x y)\alpha(t))+(-1)^{|y|(|x|+|t|)}(\alpha(t)(y x))\alpha^{2}(z)\\
&&-(-1)^{|t||x|}((x y)\alpha(t))\alpha^{2}(z)-(-1)^{|y|(|t|+|x|)+|z|(|x|+|y|+|t|)}\alpha^{2}(z)(\alpha(t)(y x))\\ &&-(-1)^{|t||y|+|z|(|x|+|y|+|t|)}\alpha^{2}(z)(\alpha(t)(x y))-(-1)^{|x|(|y|+|t|)}((y x)\alpha(t))\alpha^{2}(z)\Big)\\
&&=(-1)^{|x||y|+|z|(|x|+|y|+|t|)}[\alpha^{2}(z),\widetilde{as_{\mathcal{A}}}(y,t,x)]+(-1)^{|y|(|x|+|t|)+|z|(|x|+|y|+|t|)}[\alpha^{2}(z),\widetilde{as_{\mathcal{A}}}(t,y,x)]\\
&&+(-1)^{|t||y|+|z|(|x|+|y|+|t|)}[\alpha^{2}(z),\widetilde{as_{\mathcal{A}}}(t,x,y)]+(-1)^{|t|(|x|+|y|)+|z|(|x|+|y|+|t|)}[\alpha^{2}(z),\widetilde{as_{\mathcal{A}}}(x,t,y)]\\
&&+(-1)^{|x||t|+|z|(|x|+|y|+|t|)}[\alpha^{2}(z),\widetilde{as_{\mathcal{A}}}(x,y,t)]+(-1)^{|x|(|y|+|t|)+|z|(|x|+|y|+|t|)}[\alpha^{2}(z),\widetilde{as_{\mathcal{A}}}(y,x,t)].
\end{eqnarray*}
Note that
\begin{eqnarray}\label{crochet1}
(-1)^{|x||y|+|z|(|x|+|y|+|t|)}[\alpha^{2}(z),\widetilde{as_{\mathcal{A}}}(y,t,x)]
&=&(-1)^{|x||y|+|z|(|x|+|y|+|t|)}[\alpha^{2}(z),(y t)\alpha(x)-\alpha(y)(t x)]\nonumber\\
&=&(-1)^{|x||y|+|z|(|x|+|y|+|t|)}\Big( \alpha^{2}(z)((y t)\alpha(x))\nonumber\\
&&-\alpha^{2}(z)(\alpha(y)(t x))-(-1)^{|z|(|x|+|y|+|t|)}((y t)\alpha(x))\alpha^{2}(z)\nonumber\\
&&+(-1)^{|z|(|x|+|y|+|t|)}(\alpha(y)(t x))\alpha^{2}(z)\Big),\nonumber\\
\end{eqnarray}
and
\begin{eqnarray}\label{crochet2}
(-1)^{|y|(|x|+|t|)+|z|(|x|+|y|+|t|)}[\alpha^{2}(z),\widetilde{as_{\mathcal{A}}}(t,y,x)]
&=&(-1)^{|y|(|x|+|t|)+|z|(|x|+|y|+|t|)}[\alpha^{2}(z),(t y)\alpha(x)-\alpha(t)(y x)]\nonumber\\
&=&(-1)^{|y|(|x|+|t|)+|z|(|x|+|y|+|t|)}\Big( \alpha^{2}(z)((t y)\alpha(x))\nonumber\\
&&-\alpha^{2}(z)(\alpha(t)(y x))-(-1)^{|z|(|x|+|y|+|t|)}((t y)\alpha(x))\alpha^{2}(z)\nonumber\\
&&+(-1)^{|z|(|x|+|y|+|t|)}(\alpha(t)(y x))\alpha^{2}(z)\Big).
\end{eqnarray}
The desired condition $(\ref{id1})$ now follows from $(\ref{id2})$, $(\ref{idasso})$ and $(\ref{crochet1})$.
\end{proof}
\begin{proof} $($Theorem $(\ref{has}))$ Let $(\mathcal{A},\mu,\alpha)$ be a Hom-alternative superalgebra. To show that it is Hom-Jordan-admissible, it
suffices to prove the Hom-Jordan super-identity $(\ref{hjsuper})$ for its plus Hom-superalgebra $\mathcal{A}^{+}=(\mathcal{A},\ast,\alpha)$.\\
Using again the super-alternativity of $\widetilde{as_{\mathcal{A}}}$, this implies that
$$(\widetilde{as_{\mathcal{A}}}\circ \theta)(x y,\alpha(z),\alpha(t))=0$$
and
$$(\widetilde{as_{\mathcal{A}}}\circ \theta)(y x,\alpha(z),\alpha(t))=0$$
for any permutation $\theta$ on three letters. Since
$$[\alpha^{2}(z),(-1)^{|x||y|+|z|(|x|+|y|+|t|)}(\widetilde{as_{\mathcal{A}}}(y,t,x)+(-1)^{|y||t|}\widetilde{as_{\mathcal{A}}}(t,y,x))] = 0,$$
$$[\alpha^{2}(z),(-1)^{|t||y|+|z|(|x|+|y|+|t|)}(\widetilde{as_{\mathcal{A}}}(t,x,y)+(-1)^{|t||x|}\widetilde{as_{\mathcal{A}}}(x,t,y))] = 0,$$
and
$$[\alpha^{2}(z),(-1)^{|x||t|+|z|(|x|+|y|+|t|)}(\widetilde{as_{\mathcal{A}}}(x,y,t)+(-1)^{|x||y|}\widetilde{as_{\mathcal{A}}}(y,x,t))]=0.$$
As well, it follows from Lemma $(\ref{uti})$ that
$$\sum \limits _{x,y,t}4(-1)^{|t|(|x|+|z|)} \widetilde{as_{\mathcal{A}^{+}}}(x \ast y,\alpha(z),\alpha(t))=0,$$
from which the desired Hom-Jordan super-identity for $\mathcal{A}^{+}$ $(\ref{hjsuper})$ follows.
\end{proof}
The following constructions results are the analogues of Theorems $(\ref{induced A})$ and $(\ref{derived A})$ for Hom-Jordan and Hom-Jordan-admissible superalgebras.
\begin{thm}\label{iduced}\ \begin{enumerate}
\item Let $(\mathcal{A},\mu,\alpha)$ be a Hom-Jordan superalgebra and $\beta: \mathcal{A} \longrightarrow \mathcal{A}$ be an even Jordan superalgebra endomorphism. Then $\mathcal{A}_{\beta}=(\mathcal{A},\mu_{\beta}=\beta \circ \mu,\beta\alpha)$ is a Hom-Jordan superalgebra.
\item Let $(\mathcal{A},\mu,\alpha)$ be a Hom-Jordan superalgebra. Then the derived Hom-superalgebra \\$\mathcal{A}^{n}=(\mathcal{A},\mu^{(n)}=\alpha^{2^{n}-1} \circ \mu,\alpha^{2^{n}})$ is also a Hom-Jordan superalgebra for each $n \geq 0$.
\end{enumerate}
\end{thm}
\begin{proof} For the first assertion, first note that $\mu_{\beta}=\beta \circ \mu$ is super-commutative. To prove the Hom-Jordan super-identity $(\ref{HOM-SUPER-IDE})$ in $\mathcal{A}_{\beta}$, regard $(\mathcal{A},\mu)$ as the Hom-superalgebra $(\mathcal{A},\mu,Id)$. Then for all
$x,y,z,t \in \mathcal{H}(\mathcal{A})$ we have:
\begin{eqnarray*}
\sum \limits _{x,y,t}(-1)^{|t|(|x|+|z|)}\widetilde{ as_{\mathcal{A}_{\beta}}}(\mu_{\beta}(x,y),\beta\alpha(z),\beta\alpha(t))
&=&\sum \limits _{x,y,t}(-1)^{|t|(|x|+|z|)} \widetilde{as_{\mathcal{A}_{\beta}}}(\beta(\mu(x,y)),\beta\alpha(z),\beta\alpha(t))\\
&=&\beta^{2}\Big(\sum \limits _{x,y,t} (-1)^{|t|(|x|+|z|)} \widetilde{as_{\mathcal{A}}}(\beta(\mu(x,y)),\beta\alpha(z),\beta\alpha(t))\Big)\\
&=&\beta^{3}\Big(\sum \limits _{x,y,t} (-1)^{|t|(|x|+|z|)} \widetilde{as_{\mathcal{A}}}(\mu(x,y),\alpha(z),\alpha(t))\Big)\\
&=&0.
\end{eqnarray*}
This shows that $\mathcal{A}_{\beta}$ is a Hom-Jordan superalgebra.\\
For the second assertion, first note that $\mu^{(n)}=\alpha^{2^{n}-1} \circ \mu$ is super-commutative. To prove the Hom-Jordan super-identity $(\ref{HOM-SUPER-IDE})$ in $\mathcal{A}^{n}$, we compute
\begin{eqnarray*}
&&\sum \limits _{x,y,t}(-1)^{|t|(|x|+|z|)} \widetilde{as_{\mathcal{A}^{n}}}(\mu^{(n)}(x,y),\alpha^{2^{n}}(z),\alpha^{2^{n}}(t))\\
&=&\alpha^{2(2^{n}-1)}\Big(\sum \limits _{x,y,t}(-1)^{|t|(|x|+|z|)} \widetilde{as_{\mathcal{A}}}(\alpha^{2^{n}-1}(\mu(x,y)),\alpha^{2^{n}}(z),\alpha^{2^{n}}(t))\Big)\\
&=&\alpha^{3(2^{n}-1)}\Big(\sum \limits _{x,y,t} (-1)^{|t|(|x|+|z|)} \widetilde{as_{\mathcal{A}}}(\mu(x,y),\alpha(z),\alpha(t))\Big)\\
&=&0.
\end{eqnarray*}

This shows that $\mathcal{A}^{n}$ is a Hom-Jordan superalgebra.
\end{proof}
\subsection{Examples of Hom-Jordan superalgebras} We construct examples of Hom-Jordan superalgebras according to Theorem $(\ref{iduced})$.
\begin{exa}\label{ex-jordan1}$($Hom-Jordan superalgebra of dimension $3)$. We consider the $3$-dimensional Kaplansky superalgebra $K_{3}=\mathbb{K}e \oplus (\mathbb{K}x+\mathbb{K}y)$ (see \cite{Kapl-example3Jordan}), with characteristic of $\mathbb{K}$ is different to $2$. The   product is defined as:
$$\mu(e,e)=e,~\mu(e,x)=\frac{1}{2}x,~\mu(e,y)=\frac{1}{2}y,~\mu (x,y)=e.$$
$K_{3}$ is a simple Jordan superalgebra.\\
Even superalgebra endomorphisms $\alpha$ with respect to the basis $\{e,x,y\}$ are defined by
$$\alpha(e)=e,~~\alpha(x)=\frac{1}{c}x~~\alpha(y)=\frac{1}{c}y,$$
with $c\neq 0$.
According to  Theorem $(\ref{iduced})$, the even linear map $\alpha$ and the following multiplication
$$\mu_{\alpha}(e,e)=e,~\mu_{\alpha}(e,x)=\frac{1}{2c}x,~\mu_{\alpha}(e,y)=\frac{1}{2c}y,~[x,y]_{\alpha}=e.$$
determine a $3$-dimensional Hom-Jordan superalgebra.\\
In general, $(K_{3},\mu_{\alpha})$ is not a Jordan superalgebra. Indeed, we have
\begin{eqnarray*}
  \sum \limits _{e,e,y}(-1)^{|y|(|e|+|x|)}as_{K_{3}}(\mu_{\alpha}(e,e),x,y)
  &=& (\frac{1}{2c}-1)e \\
   &\neq&  0,
\end{eqnarray*}
for $c\neq 0, \frac{1}{2}$. Then $(K_{3},\mu_{\alpha})$ does not satisfy the Jordan super-identity.
\end{exa}
\begin{exa}\label{ex-jorda2}$($Hom-Jordan superalgebra of dimension $4)$. Let $D_{t}=(D_{t})_{0}\oplus (D_{t})_{1}$ (see \cite{Martinez-Jordan}), where $(D_{t})_{0}=span\{e_{1},e_{2}\}$ and $(D_{t})_{1}=span\{x,y\}$ be the $4$-dimensional superalgebra, $t\neq 0$, with the product given by
$$\mu(e_{i},e_{i})=e_{i},~~\mu(e_{1},e_{2})=0,~~\mu(e_{i},x)=\frac{1}{2}x,~~\mu(e_{i},y)=\frac{1}{2}y,~~[x,y]=e_{1}+te_{2},~~i=1,2.$$
This family of Jordan superalgebras $($that depend on the parameter $t)$ corresponds to the family of $17-$dimensional Lie superalgebras $D(2,1,\beta)$.\\
Superalgebra endomorphisms $\alpha$ with respect to the basis $\{e_{1},e_{2},x,y\}$ are defined by
$$\alpha(e_{1})=e_{1},~~\alpha(e_{2})=e_{2},~~\alpha(x)=a x+ by,~~\alpha(y)=c x + \frac{1+bc}{a}y.$$
According to Theorem (\ref{iduced}), the even linear map $\alpha$ and the following multiplication
$$\mu_{\alpha}(e_{i},e_{i})=e_{i},~~\mu_{\alpha}(e_{1},e_{2})=0,~~\mu_{\alpha}(e_{i},x)=\frac{1}{2}ax+\frac{1}{2}by,$$
$$~~\mu_{\alpha}(e_{i},y)=\frac{1}{2}cx+\frac{1}{2}(\frac{1+bc}{a})y,~~\mu_{\alpha}(x,y)=e_{1}+te_{2},~~i=1,2$$
where $a,b,t$ are parameter in $\mathbb{K}^{*}$ and $c \in \mathbb{K}$ satisfying $1+bc\neq 0$,
determine a $4$-dimensional Hom-Jordan superalgebra.\\
Note that $(D_{t},\mu_{\alpha})$ in general is not a Jordan superalgebra. Indeed, we have
\begin{eqnarray*}
  \sum \limits _{e_{1},e_{2},y}(-1)^{|y|(|e_{1}|+|x|)}as_{D_{t}}(\mu_{\alpha}(e_{1},e_{2}),x,y)
  &=& \Big(\frac{1}{2}-\frac{1}{2}(\frac{1+bc}{a})\Big)(e_{1}+te_{2}) \\
   &\neq&  0,
\end{eqnarray*}
for $1+bc\neq 0, a$. So, $(D_{t},\mu_{\alpha})$ does not satisfy the Jordan super-identity.
\end{exa}
\begin{thm}\label{admissible}\ \begin{enumerate}
\item Let $(\mathcal{A},\mu,\alpha )$ be a Hom-Jordan-admissible superalgebra and $\beta:\mathcal{A}\longrightarrow \mathcal{A} $ be an even Jordan-admissible superalgebra endomorphism. Then $\mathcal{A}_{\beta} = (\mathcal{A},\mu_{\beta} , \beta\alpha)$ is a Hom-Jordan-
admissible superalgebra.
\item Let $(\mathcal{A},\mu,\alpha)$ be a Hom-Jordan-admissible superalgebra. Then the derived Hom-superalgebra $\mathcal{A}^{n}=(\mathcal{A},\mu^{(n)}=\alpha^{2^{n}-1} \circ \mu,\alpha^{2^{n}})$ is also a Hom-Jordan-admissible superalgebra for each $n \geq 0$.
\end{enumerate}
\end{thm}
\begin{proof} For the first assertion, first note that the plus Hom-superalgebra $(\mathcal{A}_{\beta})^{+}=(\mathcal{A},\ast_{\beta},\beta\alpha)$ satisfies,
for all $x,y \in \mathcal{H}(\mathcal{A})$,
\begin{eqnarray*}
x \ast_{\beta} y&=&\frac{1}{2}(\mu_{\beta}(x,y)+(-1)^{|x||y|}\mu_{\beta}(y,x))\\
&=&\beta \circ \frac{1}{2}(\mu(x,y)+(-1)^{|x||y|}\mu(y,x))\\
&=&\beta \circ (x \ast y).
\end{eqnarray*}
Then $\ast_{\beta}=\beta \circ \ast$.\\
Therefore, we have $(\mathcal{A}_{\beta})^{+}=(\mathcal{A}^{+})_{\beta}$, where $\mathcal{A}^{+}$ is the Hom-Jordan-superalgebra $(\mathcal{A},\ast,\alpha)$. Since $\ast_{\beta}$ is super-commutative, it remains to prove the Hom-Jordan super-identity in $(\mathcal{A}_{\beta})^{+}=(\mathcal{A}^{+})_{\beta}$. We have
\begin{eqnarray*}
&&\sum \limits _{x,y,t}(-1)^{|t|(|x|+|z|)}\widetilde{ as_{(\mathcal{A}^{+})_{\beta}}}(\mu_{\beta}(x,y),\beta\alpha(z),\beta\alpha(t))\\
&=&\beta^{2}\Big(\sum \limits _{x,y,t}(-1)^{|t|(|x|+|z|)} \widetilde{as_{\mathcal{A}^{+}}}(\beta(\mu(x,y)),\beta\alpha(z),\beta\alpha(t))\Big)\\
&=&\beta^{3}\Big(\sum \limits _{x,y,t}(-1)^{|t|(|x|+|z|)} \widetilde{as_{\mathcal{A}}}(\mu(x,y),\alpha(z),\alpha(t))\Big)\\
&=&0.
\end{eqnarray*}
This shows that $(\mathcal{A}_{\beta})^{+}$ satisfies the Hom-Jordan super-identity, so $\mathcal{A}_{\beta}$ is Hom-Jordan-admissible superalgebra.\\
For the second assertion, first note that the plus Hom-superalgebra $(\mathcal{A}^{n})^{+}=(\mathcal{A},\ast^{(n)},\alpha^{2^{n}})$ satisfies,
for all $x,y \in \mathcal{H}(\mathcal{A})$,
\begin{eqnarray*}
x \ast^{(n)} y&=&\frac{1}{2}(\mu^{(n)}(x,y)+(-1)^{|x||y|}\mu^{(n)}(y,x))\\
&=&\alpha^{2^{n}-1} \circ \frac{1}{2}(\mu(x,y)+(-1)^{|x||y|}\mu(y,x))\\
&=&\alpha^{2^{n}-1} \circ (x \ast y).
\end{eqnarray*}
Therefore, we have $(\mathcal{A}^{n})^{+}=(\mathcal{A}^{+})^{n}$, where $\mathcal{A}^{+}$ is the Hom-Jordan superalgebra $(\mathcal{A},\ast,\alpha)$ and $(\mathcal{A}^{n})^{+}$ is its \textsf{nth} derived Hom-superalgebra. Since $\ast^{(n)}$ is super-commutative, it remains to prove the Hom-Jordan super-identity in $(\mathcal{A}^{n})^{+}=(\mathcal{A}^{+})^{n}$. We have
\begin{eqnarray*}
&& \sum \limits _{x,y,t}(-1)^{|t|(|x|+|z|)} \widetilde{as_{(\mathcal{A}^{+})^{n}}}(\mu^{(n)}(x,y),\alpha^{2^{n}}(z),\alpha^{2^{n}}(t))\\
&=&\alpha^{2(2^{n}-1)}\Big(\sum \limits _{x,y,t} (-1)^{|t|(|x|+|z|)} \widetilde{as_{\mathcal{A}^{+}}}(\alpha^{2(2^{n}-1)}(\mu(x,y)),\alpha^{2^{n}}(z),\alpha^{2^{n}}(t))\Big)\\
&=& \alpha^{3(2^{n}-1)}\Big(\sum \limits _{x,y,t} (-1)^{|t|(|x|+|z|)} \widetilde{as_{\mathcal{A}^{+}}}(\mu(x,y),\alpha(z),\alpha(t))\Big)\\
&=&0.
\end{eqnarray*}
This shows that $(\mathcal{A}^{n})^{+}$ is Hom-Jordan superalgebra, so $\mathcal{A}^{n}$ is  a Hom-Jordan-admissible superalgebra.
\end{proof}


\begin{thebibliography}{10}
\bibitem{helena-super-jor-alte-malcev-with}
H. Albuquerque,  A. Elduque and J. Laliena,
\newblock {\it Superalgebras with semisimple even part},
\newblock Communications in  Algebra and Logic
\newblock, vol. \textbf{25} (5) (1997) 1573--1587.
\bibitem{helena-classification-Malcev-super}
H. Albuquerque and  A. Elduque,
\newblock {\it Classification of Mal'tsev Superalgebras of small dimensions},
\newblock Algebra and Logic,
\newblock vol. \textbf{35} (6) (1996) 512--554.
\bibitem{ammar2010hom}
F. Ammar and A. Makhlouf,
\newblock {\it Hom-{L}ie superalgebras and {H}om-{L}ie admissible superalgebras},
\newblock Journal of Algebra
\newblock {\bf  324} (2010), 1513--1528.
\bibitem{klein-alter}
R. H. Bruck and E. Kleinfled,
\newblock {\it The structure of alternative division rings},
\newblock Proc. Am. Math. Soc.,
\newblock, vol. \textbf{2} (6) (1951) 878--890.
\bibitem{cantarini2007-jor-poisson}
N. Cantarini, and V. G. Kac,
\newblock {\it Classification of linearly compact simple Jordan and generalized Poisson superalgebras},
\newblock Journal of Algebra
\newblock {\bf  313} (2007) 100--124.
\bibitem{shest-irre-non-lie-module}
 A. Elduque and I. P. Shestakov,
\newblock {\it Irreducible non-Lie modules for Malcev superalgebras},
\newblock Journal of Algebra,
\newblock vol. \textbf{173} (1995) 622--637.
\bibitem{shest-with-trivial}
 A. Elduque and I. P. Shestakov,
\newblock {\it On Malcev superalgebras with trivial Lie nucleus},
\newblock Journal of Algebra Geometry,
\newblock vol. \textbf{} (2) (1993), 361--366.
\bibitem{elhamdadi2010deformations}
M. Elhamdadi and A. Makhlouf,
\newblock {\it Deformations of {H}om-{A}lternative and {H}om-{M}alcev
  algebras},
\newblock
\newblock (2010), To appear in Group, Algebra and Geometries,  e-Print: arXiv:1006.2499 (2010).
\bibitem{Gohr2009hom}
Gohr. A,
\newblock {\it On Hom-algebras with surjective twisting},
\newblock Journal of Algebra
\newblock   vol. {\bf 324}, no. 7  (2010) 1483--1491
\bibitem{hartwig2006deformations}
J. Hartwig, D. Larsson, and S. D. Silvestrov,
\newblock {\it Deformations of {L}ie algebras using $\sigma$-derivations},
\newblock Journal of Algebra
\newblock {\bf  295} (2006), 314--361.
\bibitem{Kapl-example3Jordan}
 I. Kaplansky,
\newblock {\it Graded Jordan algebras I},
\newblock preprint
\newblock {\bf  }.
\bibitem{Kac-Jordan-super1977}
V. G. Kac,
\newblock {\it Classification of simple $Z-$graded Lie superalgebras and simple Jordan superalgebras},
\newblock Comm. Algebra
\newblock {\bf  5} (1977), 1375--1400.
\bibitem{larsson2005quasi}
D. Larsson and S. D. Silvestrov,
\newblock {\it Quasi-{H}om-{L}ie algebras, central extensions and
  2-cocycle-like identities},
\newblock Journal of Algebra
\newblock {\bf  288} (2005), 321--344.
\bibitem{shest-rep-alter}
 M. C. Lopez-Diaz and I. P. Shestakov,
\newblock {\it Representations of exceptional simple alternative superalgebras of characteristic $3$},
\newblock Transections of the American Mathematical Society,
\newblock vol. \textbf{354} (7) (2002) 2745--2758.
\bibitem{makhlouf hom-altern2009hom}
A. Makhlouf,
\newblock {\it Hom-alternative algebras and {H}om-{J}ordan algebras},
\newblock Int. Electron. J. Algebra
\newblock {\bf  8} (2010), 177--190.
\bibitem{makhloufand silvestrov2006hom}
A. Makhlouf and S. D. Silvestrov,
\newblock {\it On {H}om-algebra structures},
\newblock {J. Gen. Lie Theory Appl. 2},
\newblock {\bf 2} (2008), 51--64.
\bibitem{makhlouf2007notes of}
A. Makhlouf and S. D. Silvestrov,
\newblock {\it Notes on formal deformations of {H}om-associative and
  {H}om-{L}ie algebras},
\newblock
\newblock Forum Math, vol. \textbf{22} (4) (2010) 715--759.
\bibitem{Martinez-Jordan}
C. Martinez,
\newblock {\it Simplicity of Jordan superalgebras and relations with Lie structures},
\newblock Irish Math. Soc. Bulletin
\newblock {\bf  50} (2003), 97--116.
\bibitem{Martinez-Repre-theory-Jordan}
C. Martinez and E. Zelmanov,
\newblock {\it Representations theory of Jordan superalgebras I},
\newblock Trans. of the Amer. Math. Soc.
\newblock {\bf  362} (2) (2010), 815--846.
\bibitem{shest-Non-commutative}
 A. P. Pozhidaev and I. P. Shestakov,
\newblock {\it Noncommutative Jordan superalgebras of degree $n> 2$ },
\newblock Algebra and Logic
\newblock, vol. \textbf{49} (1) (2010) 512--543.12.
\bibitem{Racine-classi-jordan-super}
M. Racine and E. Zelmanov,
\newblock {\it Classification of simple Jordan superalgebras with semisimple even part},
\newblock Journal of Algebra
\newblock {\bf  270} (2), (2003) 374--444.
\bibitem{Richardson-quadr-jorda99}
P. A. Richardson,
\newblock {\it Centroids of quadratic Jordan superalgebras},
\newblock arXiv
\newblock , vol. \textbf{}  (2006) .
\bibitem{shest-non-lie-module}
I. P. Shestakov and A. Elduque,
\newblock {\it Prime non-Lie modules for Mal'Tsev superalgebras},
\newblock Algebra and Logic,
\newblock vol. \textbf{33} (4) (1994) 253--262.
\bibitem{shest-jordan-super}
I. P. Shestakov,
\newblock {\it Alternative and Jordan superalgebras},
\newblock Siberian Advances in Mathematics
\newblock, vol. \textbf{9} (1999) 83--99.
\bibitem{shest-poisson}
I. P. Shestakov,
\newblock {\it Quantization of Poisson superalgebras and speciality of Jordan Poisson superalgebras},
\newblock Algebra and Logic,
\newblock vol. \textbf{32} (5) (1993) 309--317.
\bibitem{shest-zelm}
I. P. Shestakov and E. Zalmanov,
\newblock {\it Prime alternative superalgebras and the nilpotency of the radical of a free alternative algebras},
\newblock Izv. Akad. Nauk. SSSR
\newblock, vol. \textbf{54}  (1990) 676--693.
\bibitem{sheng2010representations}
Y. Sheng,
\newblock {\it Representations of {H}om-{L}ie algebras},
\newblock {Algebras and Representation Theory},
\newblock (2010), 1--18.
\bibitem{shtern}
A. S. Shtern,
\newblock {\it Representation of an exceptional Jordan superalgebra},
\newblock Funktsional Anal. i Prilozhen
\newblock, vol. \textbf{21}  (1987) 93--94.
\bibitem{TRU-rep-alter}
 M. N. Trushina,
\newblock {\it Representations of  simple superalgebra of  $B(1,2)$},
\newblock Journal of Mathematical Sciences
\newblock, vol. \textbf{114} (2) (2003) 512--554.
\bibitem{Yau:EnvLieAlg} D. Yau,
\newblock \emph{Enveloping algebra of Hom-Lie algebras},
\newblock
\newblock J. Gen. Lie Theory Appl. \textbf{2} (2008) 95--108.
\bibitem{Yau:homology} D. Yau,
\newblock \emph{ Hom-algebras and homology,}
\newblock J. Lie Theory \textbf{19} (2009) 409--421.
\bibitem{Yau2012}
D. Yau,
\newblock {\it Hom-maltsev, Hom-alternative and Hom-jordan algebras},
\newblock International Electronic Journal of Algebra
\newblock {\bf  11} (2012) 177--217.
\bibitem{Zelm and she-alter}
E. I. Zel'manov and I. P. Shestakov,
\newblock \emph{Prime alternative superalgebras and nilpotence of free alternative algebra,}
\newblock  Math USSR Izvestiya. \textbf{37} (1991) UDC 512--554.
\bibitem{ZhangBai} R. Zhang, D. Hou and C. Bai,
\newblock \emph{A Hom-version of the affinizations of Balinskii-Novikov and
Novikov superalgebras,}
\newblock  J. Math. Phys. 52 (2011) 023505.\\\\\\\\



\end{thebibliography}
\end{document}